\documentclass[a4paper]{article}

 \usepackage[english]{babel}
 \usepackage[utf8x]{inputenc}
 \usepackage[T1]{fontenc} 

\usepackage[a4paper,top=3cm,bottom=2cm,left=3cm,right=3cm,marginparwidth=1.75cm]{geometry}

\usepackage{amsmath}
\usepackage{amsthm}
\usepackage{amssymb}
\usepackage{comment}
\usepackage{cases}
\usepackage{graphicx}
\usepackage[colorinlistoftodos]{todonotes}
\usepackage[colorlinks=true, allcolors=blue]{hyperref}
\usepackage{empheq} 
\usepackage[most]{tcolorbox} 

\usepackage{subcaption}
\usepackage{xcolor}
\usepackage{comment}
\usepackage[font={small}]{caption}
\usepackage[normalem]{ulem} 
\newcommand\str{\bgroup\markoverwith
{\textcolor{red}{\rule[0.5ex]{2pt}{1.5pt}}}\ULon} 

\newcommand{\al}{\alpha}
\newcommand{\be}{\beta}
\newcommand{\la}{\lambda}
\newcommand{\cla}{c_\lambda}
\newcommand{\dla}{d_\lambda}
\newcommand{\Dla}{\Delta_\lambda}

\newcommand{\id}{I_d}
\newcommand{\diag}{\mathbf{diag}}
\newcommand{\zeros}{0_d}

\newcommand{\norm}[1]{\left\|#1\right\|}

\newcommand{\bigO}{\mathcal{O}}

\newcommand{\R}{{\mathbb{R}}}

\DeclareMathOperator*{\argmin}{arg\,min}

\DeclareMathOperator{\Tr}{Tr}

\newcommand{\calS}{\mathcal{S}}

\newcommand{\ty}{\tilde{y}}

\newcommand{\tx}{\tilde{\xi}}

\newcommand{\tiR}{\tilde{R}}
\newcommand{\tla}{\tilde{\lambda}}
\newcommand{\baf}{\bar{f}}

\newcommand{\E}{\mathbb{E}}
\newcommand{\Sml}{$f\in S_{\mu,L}(\mathbb{R}^d)$}
\newcommand{\uab}{u_{\alpha,\beta}}
\newcommand{\buab}{\bar{u}_{\alpha,\beta}}
 \newtheorem{theorem}{Theorem}[section]  
 
 \newtheorem{ex}[theorem]{Example}
\newtheorem{proposition}[theorem]{Proposition}
 \newtheorem{lemma}[theorem]{Lemma}
 \newtheorem{rmk}[theorem]{Remark}
\newtheorem{corollary}[theorem]{Corollary}
 
 \newtheorem{asmp}[theorem]{Assumption}

\newcommand{\beq}{\begin{equation}}
\newcommand{\eeq}{\end{equation}}
\newcommand{\beqa}{\begin{eqnarray}}
\newcommand{\eeqa}{\end{eqnarray}}
\newcommand{\beqs}{\begin{equation*}}
\newcommand{\eeqs}{\end{equation*}}
\newcommand{\beqas}{\begin{eqnarray*}}
\newcommand{\eeqas}{\end{eqnarray*}}

\newtcbox{\mymath}[1][]{%
    nobeforeafter, math upper, tcbox raise base,
    enhanced, colframe=blue!30!black,
    colback=blue!30, boxrule=1pt,
    #1}
    
\def\R{\mathbb{R}}

\title{Robust Accelerated Gradient Methods for Smooth Strongly Convex Functions}

\author{
\begin{tabular}[t]{c@{\extracolsep{2em}}c @{\extracolsep{2em}}c} 
\textbf{Necdet Serhat Aybat}\footnote{Authors are arranged in alphabetical order.} & \textbf{Alireza Fallah}$^*$\\
Pennsylvania State University & Massachusetts Institute of Technology \\
University Park, PA, USA & Cambridge, MA, USA \\
nsa10@psu.edu & afallah@mit.edu \\ \\ 
\textbf{Mert G\"urb\"uzbalaban}$^*$ & \textbf{Asuman Ozdaglar}$^*$ \\
Rutgers University & Massachusetts Institute of Technology\\
Piscataway, NJ, USA & Cambridge, MA, USA \\
mg1366@rutgers.edu & asuman@mit.edu
\end{tabular}
}

\date{}

\begin{document}
\maketitle

\begin{abstract} We study the trade-offs between convergence rate and robustness to gradient errors in designing a first-order algorithm. We focus on gradient descent (GD) and accelerated gradient (AG) methods for minimizing strongly convex functions when the gradient has random errors in the form of {additive white noise}. With gradient errors, the function values of the iterates need not converge to the optimal value; hence, we define the robustness of an algorithm to noise as the asymptotic expected suboptimality of the iterate sequence to input noise power. For this robustness measure, we provide exact expressions for the quadratic case using tools from robust control theory and tight upper bounds for the smooth strongly convex case using Lyapunov functions certified through matrix inequalities. We use these characterizations within an optimization problem which selects parameters of each 
algorithm to achieve a particular trade-off between rate and robustness. Our results show that AG can achieve acceleration while being more robust to random gradient errors. This behavior is quite different than previously reported in the deterministic gradient noise setting. We also establish some connections between the robustness of an algorithm and how quickly it can converge back to the optimal solution if it is perturbed from the optimal point with deterministic noise. Our framework also leads to practical algorithms that can perform better than other state-of-the-art methods in the presence of random gradient noise. 
\end{abstract}

\section{Introduction}
For many large-scale convex optimization 
and machine learning problems
, first-order methods have been the leading computational approach for computing low-to-medium accuracy solutions because of their cheap iterations and mild dependence on the problem dimension and data size. The typical analysis of first-order methods assumes the availability of exact gradient information and provides statements on the rate of convergence to the optimal solution as the main performance criterion. However, in many applications, the gradient contains deterministic or stochastic errors either because the gradient is computed by inexactly solving an auxiliary problem~\cite{aspremontSmooth08,devolder2014first}, or the method itself involves errors with respect to the full gradient as in standard incremental gradient, stochastic gradient, and stochastic approximation methods \cite{PolyakJuditskyAcceleration,robbins1951stochastic,bertsekas1999nonlinear,bertsekas2011incremental}. When there are persistent errors in gradients, the iterates do not converge and could oscillate in a neighborhood of the optimal solution or may even diverge \cite{bertsekas1999nonlinear,bertsekas2011incremental,devolder2014first,flammarion2015averaging}. This makes robustness of the algorithms to gradient errors (in terms of solution accuracy) another important performance objective \cite{devolder2014first,Hardt-blog}. In particular, even though accelerated gradient method proposed by Nesterov converges faster than gradient descent (GD) in the absence of noise for convex problems \cite{nesterov1983method}, 
it was shown that they are less robust to errors, i.e., accelerated methods require higher precision gradient information than 
GD to achieve the same solution accuracy \cite{devolder2014first,aspremontSmooth08,flammarion2015averaging,Schmidt11InexactProx}.

In this paper, we study the trade-offs between convergence rate and robustness to gradient errors in designing a first-order algorithm. 
We focus on 
GD and Nesterov's accelerated gradient (AG) method for minimizing strongly convex smooth functions when the gradient has stochastic errors and investigate how the parameters of each 
algorithm should be set to achieve a particular trade-off between these two performance objectives. To study this question systematically, we employ tools from control theory whereby we represent each of the algorithms as a dynamical system. This approach has attracted recent attention and has already led to a number of insights for the design and analysis of optimization algorithms \cite{lessard2016analysis,hu2017dissipativity,cyrus2017robust,fazlyab2017analysis,wilson2016lyapunov,hu2017approxSG}. The novelty of our work is to use this 
approach to provide explicit characterizations of robustness, which can then be placed in a computationally tractable 
 optimization problem for selecting the algorithm parameters 
 to systematically achieve a desired trade-off.

We first focus on 
problems with a strongly convex quadratic objective function. For this case, the rate of convergence of any of the two algorithms we study is given by the spectral radius of the ``state-transition" matrix in the dynamical system representation. To characterize robustness, we consider the asymptotic expected suboptimality for the centered iterate sequence (output vector of the dynamical system) per unit noise which is a measure of the asymptotic accuracy of the iterates. For the quadratic case we show that this limit 
exists and can be 
characterized using the \textit{$H_2$ norm} of a transformed linear dynamical system. The $H_2$ norm is a fundamental measure for quantifying robustness of a linear system to noise~\cite{zhou1996robust} and admits various definitions and characterizations. We focus on a particular representation of the $H_2$ norm that requires
the solution of a discrete Lyapunov equation. This representation 
leads to explicit expressions for robustness of GD and AG. 

Using this result, we study the rate and robustness trade-off of 
the GD method for minimizing quadratic strongly convex functions. The spectral radius of the state-transition matrix corresponding to GD dynamics, hence, the rate of convergence for GD, can be expressed in terms of the smallest and largest eigenvalues of the positive definite matrix $Q$ defining the Hessian of the strongly convex quadratic objective. We show that our robustness measure admits a tractable characterization for GD in terms of the 
spectrum of $Q$. 
We also show 
a fundamental lower bound on the robustness level of an algorithm for any achievable convergence rate. 

We next consider 
the AG method defined 
by two parameters: stepsize $\alpha$ and momentum parameter $\beta$. 
Our first step is to characterize the \textit{stability region} of the method, i.e., the set of nonnegative $(\alpha, \be)$ for which the spectral radius of the state-transition matrix is less than or equal to one. Similar to GD, we then provide an explicit characterization of the $H_2$ norm of 
the dynamical system representation of AG. We use these explicit expressions for both GD and AG within an optimization problem for selecting the parameters to minimize the robustness measure subject to a given upper bound on the convergence rate. Our results show that AG with properly selected parameters is superior to GD in the sense that AG can achieve the same rate with GD while being more robust to noise; similarly, AG can be tuned to be faster than GD while achieving the same robustness level. This behavior 
contrasts with the comparison of GD and AG in the deterministic gradient error setting in \cite{devolder2014first}, 
which shows 
GD performance degrades gracefully while AG may accumulate error.
These results show the random 
and deterministic noise settings have different behavior.

In our second set of results, we extend our analysis to handle minimization of strongly convex smooth functions, i.e., $\min_{x\in\R^d}f(x)$. 
In this setting, the dynamical system representation of a first-order algorithm will no longer be a linear system due to the nonlinear gradient map, $\nabla f$. The analysis in this section is not limited to GD or AG; in particular, given a first-order optimization algorithm, we use a linear dynamic system with {\it nonlinear} feedback to model the dynamic behavior of the algorithm. 
For these systems, we again use a robustness measure that can be seen as 
a discrete-time version of a more general $H_2$-norm 
for nonlinear systems~\cite{fleming1995risk} --- see also ~\cite{Stoorvogel93} for a similar definition given for linear systems with nonlinear feedback. We derive upper bounds on the robustness measures for GD and AG using Lyapunov functions certified through matrix inequalities and investigate the trade-off between rate and robustness.


In addition to the above cited papers, 
Devolder's Ph.D. thesis \cite{devolder2013thesis} is closely related to our paper. Chapters 4 and 6 of this thesis, considered smooth weakly convex functions under a deterministic oracle model 
whereas Chapter 7 focused on a stochastic oracle model; 
these general oracles can model inexactness in 
the gradients as well as function evaluations. 
In the deterministic oracle case, Devolder shows that primal gradient method (PGM) and the dual gradient method (DGM) on smooth weakly convex objectives exhibit slow convergence with a rate $\mathcal{O}(1/k)$ but without accumulation of errors (the total effect of errors after $k$ iterations is equal to the individual error $\delta$ of each first-order information); whereas accelerated gradient methods converge faster with rate $\mathcal{O}(\frac{1}{k^2})$ but suffers from accumulation of errors at a linear rate $\mathcal{O}(k\delta)$. Based on these observations, Devolder et al. \cite{devolder2013intermediate} design a novel family of first-order methods called intermediate gradient methods (IGM) for solving smooth weakly convex problems; these methods have an intermediate speed and intermediate sensitivity to gradient errors, i.e., faster than classical gradient methods and more robust to noise than the accelerated gradient methods.  
In the stochastic oracle case, Devolder developed a class of accelerated gradient methods for weakly convex functions with decaying stepsize rules and showed that the expected suboptimality admits the convergence rate $\mathcal{O}\left( \frac{LR^2}{k^2} + \frac{\sigma R}{\sqrt{k}}\right)$ as opposed to the $\mathcal{O}\left( \frac{LR^2}{k} + \frac{\sigma R}{\sqrt{k}}\right)$ rate of PGM and DGM, where $R$ is the distance of the initial point to the optimal solution, $L$ is the Lipschitz constant for the gradient of the objective $f(x)$ and $\sigma$ is the level of the stochastic noise \cite[Ch. 7]{devolder2013thesis}. In his thesis, Devolder studied also smooth and strongly convex objectives under the same deterministic oracle model, showing that both PGM and DGM converge with a rate that is proportional to $\exp(-k \frac{\mu}{L})$ without accumulation of errors where $\mu$ is the strong convexity constant, whereas accelerated gradient converges faster proportional to $\exp(-k \sqrt{\frac{\mu}{L}})$ while the error accumulation behaves like $\sqrt{\frac{L}{\mu}}\delta$ up to a constant \cite[Chapter 5]{devolder2013thesis}. On the other hand, the smooth and strong convex objectives subject to stochastic errors was left as future work \cite[Ch. 8.1.1]{devolder2013thesis}; 
and this is the setting considered in our paper where we focus on \emph{stochastic} additive gradient errors for \emph{strongly convex} objectives, which arises in a number of problems in machine learning and large-scale optimization
\cite{Hardt-blog,bassily2014private,raginsky2017non}.
In this setting, Ghadimi and Lan~\cite{ghadimi2012optimal,ghadimi2013optimal} propose an accelerated
method called AC-SA for solving strongly convex composite optimization problems obtaining an optimal rate matching the lower
complexity bounds for stochastic optimization. 
{Flammarion and Bach \cite{flammarion2015averaging} considered accelerated versions of gradient descent for quadratic optimization 
that attain the optimal rates for both the bias and variance terms, respectively, in the performance bounds}. Michalowsky and Ebenbauer \cite{michalowsky2014multidimensional} posed the design of deterministic gradient algorithms as a state feedback problem and used robust control theory and linear matrix inequalities to 
study them. Mohammadi et al. \cite{mohammadi2018variance} 
examined the sensitivity of accelerated algorithms to stochastic noise for strongly convex quadratic functions in terms of the steady-state variance of the optimization variable. Finally, Dvurechensky et al. \cite{dvurechensky2016stochastic} consider composite convex optimization problems with inexact first-order oracles having both deterministic and stochastic errors; indeed, their inexact oracle is an extension of the one adopted in~\cite{devolder2013intermediate, devolder2014first} to include stochastic errors. For this setting, Dvurechensky et al. \cite{dvurechensky2016stochastic} propose a stochastic version of the intermediate gradient method in~\cite{devolder2013intermediate} and analyze the convergence rate in terms of expected suboptimality and error accumulation due to inexact oracle; the proposed algorithm in~\cite{devolder2013intermediate} has complexity bounds 
matching the optimal lower complexity bounds for composite convex problems with stochastic inexact oracle as in~\cite{ghadimi2012optimal,ghadimi2013optimal}.
Finally, Hu et al. \cite{hu2017approxSG} analyze the stochastic gradient method under deterministic noise and study the effect of the stepsize on the convergence rate and the asymptotic neighborhood of convergence.
\color{black}
These papers focus on convergence rate of the algorithms, whereas our goal is to define robustness and design algorithms to successfully trade-off different objectives. Furthermore, we make some connections between the robustness of a first-order method and its behavior when perturbed from the optimal solution and  show that AG is more resilient to perturbations 
 {in the sense that it} recovers the optimal point with less energy compared to GD for sufficiently small stepsizes.
We will also demonstrate in our numerical experiments that the framework we propose is competitive in practice with the existing state-of-the-art algorithms from the literature and can outperform them in some problems, illustrating the potential of the proposed framework in practice. In fact, {in a companion paper, we use our framework to develop a universally} 
optimal multi-stage stochastic gradient algorithm for stochastic optimization \cite{aybat2019universally}
which achieves the lower bounds without 
assuming a known bound for suboptimality or the variance of the gradient noise.
\subsection{Preliminaries and Notation}
For two functions $g,h$ defined over positive integers, we say $f = \Theta(g)$ if there exist constants $C_l, C_u$ and $n_0$ such that $C_l g(n) \leq f(n) \leq C_u g(n)$ for every positive integer $n\geq n_0$. For a set $I$, $|I|$ denotes the cardinality of the set $I$.
Let $\delta[k]$ denote the Kronecker delta function, i.e., $\delta[0]=1$ and $\delta[k]=0$ for any 
integer $k\geq 1$. The $d\times d$ identity and zero matrices are denoted by $\id$ and $0_d$, respectively. We define $\diag(a_1,...,a_d)$ or $\diag([a_i]_{i=1}^d)$ as the diagonal matrix with diagonal entries $a_1,...,a_d$; similarly, $\diag([A_i]_{i=1}^d)$ denotes a block diagonal matrix with $i$-th block equal to $A_i\in\R^{n_i\times n_i}$ for $i=1,\ldots, d$. For matrix $A \in \mathbb{R}^{d\times d}$, $\Tr(A)$ denotes the trace of $A$. We use the superscript $^\top$ to denote the transpose of a vector or a matrix depending on the context. The spectral radius of $A$ is defined as the largest absolute value of its eigenvalues and is denoted by $\rho(A)$. We say that a square matrix $A$ is \emph{discrete-time stable}, if all of its eigenvalues lie strictly inside the unit disc in the complex plane, i.e., if $\rho(A)<1$. Throughout this paper, all vectors are represented as column vectors. Let $\mathbb{S}^m$ be the set of all symmetric $m\times m$ matrices. Similarly, $\mathbb{S}^m_{++}$ ($\mathbb{S}^m_{+}$) denote the set of all symmetric and positive (semi)-definite $m \times m$ matrices. For two matrices $A \in \mathbb{R}^{m \times n}$ and $B \in \mathbb{R}^{p \times q}$, their Kronecker product is denoted by $A\otimes B$.
For scalars $0<\mu\leq L$, we define $S_{\mu, L}(\mathbb{R}^d)$ as the set of 
continuously differentiable functions $f:\mathbb{R}^d\rightarrow \mathbb{R}$ that are strongly convex with modulus $\mu$ and have Lipschitz-continuous gradients with constant $L$, i.e., 
{
\begin{align}
\frac{L}{2} \Vert x-y \Vert^2 \geq f(x)-f(y)-\nabla f(y)^\top (x-y)&\geq \frac{\mu}{2} \Vert x-y \Vert^2,\quad \forall~x,y \in \mathbb{R}^d,
\label{ineq_S}
\end{align}}%
(see e.g. \cite{nesterov_convex}) where the gradient $\nabla f$ is represented as a column vector. 
The ratio $\kappa \triangleq \frac{L}{\mu}$ is called the \emph{condition number} of $f$. 
In many places, we also use the following relation for strongly convex smooth functions. 
\begin{lemma}[Theorem 2.1.12 in \cite{nesterov_convex}] \label{stronglyconvex_lem}
If \Sml, then for every $x,y \in \mathbb{R}^d$,
{
\begin{equation*}
(\nabla f(x) - \nabla f(y))^\top (x-y) \geq \frac{\mu L}{\mu +L} \|x-y\|^2 + \frac{1}{\mu+L} \|\nabla f(x) - \nabla f(y)\|^2.
\end{equation*}}%
\end{lemma}
For our subsequent analysis, 
we represent the preceding relation in matrix form: 
{
\begin{equation}\label{GD_X}
\begin{bmatrix} 
	x-y\\
    \nabla f(x) - \nabla f(y)
\end{bmatrix}^\top
\begin{bmatrix} 
	2\mu L I_d & -(\mu+L) I_d\\
    -(\mu+L) I_d & 2 I_d
\end{bmatrix}
\begin{bmatrix} 
	x-y\\
    \nabla f(x) - \nabla f(y)
\end{bmatrix} \leq 0, \quad \forall\ x,y\in\R^d.
\end{equation}}%
\section{Optimization Algorithms as Dynamical Systems}
\label{sec-dyn-system-intro}
Our goal 
is to design first-order algorithms with certain rate-robustness balance to solve
\begin{equation}\label{main-opt-prob}
f^*\triangleq \min_{x \in \mathbb{R}^d} f(x),\quad \hbox{where}\quad f\in S_{\mu,L}(\mathbb{R}^d),
\end{equation}
when the gradient $\nabla f$ is corrupted by random errors in the form of additive white noise. We denote the unique optimal solution of problem \eqref{main-opt-prob} by $x^*$. 
We will focus on Gradient Descent (GD) and Accelerated Gradient Descent (AG) 
and show how the parameters of these algorithms can be tuned to optimize various performance metrics.

Our analysis builds on a dynamical system representation of these algorithms. A discrete-time dynamical system with a feedback rule $\phi$ can be expressed as 
\begin{align}
\label{dyn_sys:main}
\xi_{k+1} = A \xi_k + B u_k,\quad  
y_k = C \xi_k + D u_k,\quad  
u_k  = \phi(y_k), 
\end{align}
for $k\geq 0$, where $\xi_k \in \mathbb{R}^m$ is the \emph{state}, $u_k\in\mathbb{R}^d$ is the \emph{input}, and $y_k \in \mathbb{R}^d$ is the \emph{output}. The matrices $A, B, C$ and $D$ are called the \emph{system matrices}; they are fixed matrices with appropriate dimensions. The function $\phi:\R^d \to \R^d$ defines the feedback rule that relates the output of this system to its input.

Consider the GD method for solving problem \eqref{main-opt-prob}. Given $x_0\in \mathbb{R}^d$, the GD iterations with a constant stepsize $\alpha>0$ take the following form for $k\geq 0$:
\begin{equation} \label{gradient_update} 
x_{k+1} = x_k - \alpha \nabla f(x_k),
\end{equation}
which can be cast as \eqref{dyn_sys:main} by setting $\xi_k =x_k$, $\phi(\cdot)=\nabla f(\cdot)$ and letting 
\begin{equation}\label{gradient_ABCD}
A=\id, \quad B=-\alpha\id,\quad C=\id, \quad D=0_d.
\end{equation}
On the other hand, when implemented on \eqref{main-opt-prob}, the AG method with constant stepsize $\alpha>0$ and momentum parameter $\beta>0$
generates the iterates as follows for $k\geq 0$:
\begin{equation}\label{nest:main}
y_k = (1+\beta)x_k - \beta x_{k-1}, \quad x_{k+1} = y_k - \alpha \nabla f(y_k). 
\end{equation}
Setting $\phi(\cdot)=\nabla f(\cdot)$ and defining the state vector
$\xi_k = \begin{bmatrix} 
	x_k^\top & x_{k-1}^\top
\end{bmatrix}^\top$,
AG iterations can be rewritten as in \eqref{dyn_sys:main} for
{
\begin{equation} \label{nest_ABCD} 
A = \begin{bmatrix} 
	(1+\be)\id & -\be\id \\ 
    \id 	  &  \zeros 
\end{bmatrix},\quad 
B = \begin{bmatrix} 
	-\al \id  \\ 
     \zeros 
\end{bmatrix},\quad 
C = \begin{bmatrix} 
	(1+\be)\id & -\be\id  
\end{bmatrix},\quad 
D=0_d.
\end{equation}}%
For both algorithms, the iterates $x_k$ are captured by the state $\xi_k$ of the dynamical system representation. 

In this work, we assume that at each iteration $k\geq 0$, instead of the actual gradient $\nabla f(y_k)$, we have access to a noisy version $\nabla f(y_k) + w_k$ where $w_k \in \R^d$ represents the additive noise.
In the dynamical system representation, the noisy iterations of the GD and AG algorithms could be written as
\begin{align}\label{noisy-dyn_sys:main}
\xi_{k+1} = A \xi_k + B (u_k + w_k), \quad
y_k  = C \xi_k, \quad
u_k = \nabla f(y_k),
\end{align}
where $A$, $B$, $C$, and $D$ are selected according to \eqref{gradient_ABCD} for GD or \eqref{nest_ABCD} for AG.\footnote{Although our focus in this paper will be primarily on GD and AG dynamics under noise, it will be clear from our discussion that our ideas naturally extend to many other algorithms that admit such a dynamical system representation including the heavy-ball and the robust momentum methods \cite{cyrus2017robust,lessard2016analysis,hu2017dissipativity}.}
Except for Section~\ref{sec:stability} where we study deterministic perturbations, we assume throughout this paper that the sequence $\{w_k\}_k$ of random variables satisfies the following assumption.
\begin{asmp}\label{asmp1}
For any $k \geq 0$, the random variable $w_k$ in \eqref{noisy-dyn_sys:main} is zero mean and independent from $\{\xi_i\}_{i=1}^k$ and $\{y_i\}_{i=1}^k$. In addition, there exists a scalar $\sigma>0$ such that $\E(w_k w_k^\top) =\sigma^2 \id$ for any $k \geq 0$.  
\end{asmp}
This noise structure arises naturally in stochastic optimization where the full gradient is approximated from finitely many samples (see e.g. \cite{PolyakJuditskyAcceleration}), in regression problems \cite{flammarion2015averaging,bachGD,bach-non-strongly-cvx} as well as in optimization algorithms where the full gradient is subject to an isotropic noise or uncertainty (see e.g. \cite{jadbabaie2013combinatorial,jadbabaie2016performance}). The special case when $w_k$ is Gaussian also appears in algorithms where random noise is intentionally injected to the gradient to guarantee privacy (e.g. \cite{bassily2014private}) or to ensure global convergence, as in the Euler-Mariyama discretization of the overdamped and underdamped Langevin dynamics \cite{BartlettUnderdampedLangevin17,Eberle17,GGZ-underdamped-18,gao2018breaking}. It will be clear from our discussion that our results can be extended to the {\it structured noise} case, i.e., when the covariance matrix $\E(w_k w_k^\top) = S$ for some positive definite matrix $S$.

Consider a first-order algorithm (e.g., GD or AG) subject to additive noise satisfying Assumption~\ref{asmp1}. For this scenario, where the noise is {\it persistent}, i.e., it does not decay over time, it is possible that $\lim_{k\rightarrow\infty}\mathbb{E}[f(x_k)]$ may {\it not} exist; therefore, one natural way of defining {\it robustness} of an algorithm to noise is to consider the worst-case limiting suboptimality along all possible subsequences, i.e., 
\begin{equation} \label{eq:robust_def}
\mathcal{J} \triangleq \limsup\limits_{k\to\infty} \frac{1}{\sigma^2}\E[f(x_{k}) - f^*]. 
\end{equation}
Clearly, $\mathcal{J}$ depends on the choice of algorithm parameters. Moreover, since the limit  $\lim_{k\rightarrow\infty}f(x_k)$ may not exist when the gradients are perturbed by persistent additive noise, both notions of ``convergence'' and ``convergence rate'' are vague. To make these terms more precise in our context, consider the line segment $[f^*,\  f^*+\sigma^2\mathcal{J}]$. In the subsequent sections of the paper, we show that $\{f(x_k)\}_{k\geq 0}$ sequence converges to this line segment {\it linearly} with a rate depending on the algorithm parameters. Thus, the aim of this paper is to investigate this trade-off between the robustness and rate associated with a given first-order algorithm and to understand the dependence of these key notions of convergence on the choice of algorithm parameters. We believe that achieving this goal would provide an important leverage to decision makers to set the parameters in such a way that fits the purpose of the application. We focus on the expected suboptimality $\{\E[f(x_k) - f^*]\}_k$ in the text since this is typically the object of study in the literature for quantifying the performance of similar algorithms.

It is worth emphasizing that 
{robustness} can also be studied in the solution space. Indeed, let $\{x_k\}_{k\geq 0}$ be a random iterate sequence corresponding to \eqref{noisy-dyn_sys:main} where $\{w_k\}_k$ models the additive noise sequence and satisfies Assumption \ref{asmp1}. Due to the noise injected at each step, the sequence $\{x_k\}$ will oscillate around the optimal solution with a non-zero variance; therefore, 
another natural metric to measure robustness is the worst-case limiting distance to the optimal solution $x^*$ along all possible iterate subsequences, 
i.e., 
\begin{equation} \label{eq:robust_def_alternative}
\mathcal{J}' \triangleq \limsup_{k\to \infty}\limits 
\frac{1}{\sigma^2}\E[\norm{x_k-x^*}^2].
\end{equation}%
Similarly, the convergence rate could be defined to be the linear rate that $\{x_k\}_{k\geq 0}$ converges to the ball $\{x\in\R^d:\ \norm{x-x^*}^2\leq \sigma^2\mathcal{J}'\}$. The quantity $\mathcal{J}'$ can be viewed as the \emph{robustness to noise in terms of iterates} because it is 
equal to the ratio of the power of the iterates to the power of the input noise, measuring how much a system amplifies input noise. In particular, the smaller this measure is, the more robust a system is under additive random noise.\footnote{See Appendix \ref{iterate_case}, provided as a supplementary material, where we derive robustness results 
based on $\mathcal{J}'$ for both GD and AG.} In Section~\ref{quad_case}, we 
{remark} that $\mathcal{J}'$ is indeed the $H_2$ norm of the dynamical system in \eqref{noisy-dyn_sys:main} with $C=I$, a notion being applicable to both linear and non-linear systems \cite{Stoorvogel93,fleming1995risk}. Later in Section~\ref{sec:stability}, we will use $\mathcal{J}'$ to make some connections between the robustness of a first-order method with its behavior when perturbed from the optimal solution.

\section{Quadratic Functions}\label{quad_case} In order to understand the effect of noise on the dynamics, we find it insightful to first focus on the case where the objective function is quadratic.
Let $f\in S_{\mu,L}(\mathbb{R}^d)$ be a quadratic function given by $f(x) = \tfrac{1}{2} x^\top Q x - p^\top x + r$ where $Q$ is symmetric and positive definite with eigenvalues $\{\lambda_i\}_{i=1}^d$ listed in increasing order satisfying $0< \mu = \lambda_1 \leq \lambda_2 \leq \dots \leq \lambda_d = L$. The gradient of $f$ is given by 
\begin{equation}\label{grad_formula}
\nabla f(x) =  Q x - p = Q(x-x^*),
\end{equation}
where $x^* = Q^{-1} p$ is the optimal solution to problem \eqref{main-opt-prob}. Plugging the formula for the gradient $\nabla f(y_k)$ from \eqref{grad_formula} into \eqref{noisy-dyn_sys:main}, we obtain
\begin{align}\label{dyn_sys_quad:main} 
\xi_{k+1} = (A+BQC)\xi_k-BQx^* + B w_k,\quad y_k = C \xi_k.
\end{align}
With $\xi^*$ equal to $x^*$ for GD and $[x^{*^\top} x^{*^\top}]^\top$ for AG, in both cases we have $\xi^* = A\xi^*$ and $x^* = C \xi^*$ where $A$ and $C$ are given in \eqref{gradient_ABCD} and \eqref{nest_ABCD} for GD and AG, respectively.
Therefore, defining $\ty_k \triangleq y_k-x^*$ and $\tx_k \triangleq \xi_k - \xi^*$,
\eqref{dyn_sys_quad:main} yields
\begin{align}\label{dyn_sys_quad2:main}
\tx_{k+1} = A_Q \tx_k + B w_k,\quad \ty_k = C \tx_k,
\end{align}
where $A_Q$ is the state-transition matrix given by $A_Q=A+BQC$.

In the absence of noise (when $w_k=0$ for all $k$), if $\rho(A_Q)$ is less than one, then we clearly have $\tx_k \to 0$ and $\ty_k \to 0$ {\it linearly}. As a consequence, the 
suboptimality, 
$f(x_k)-f^*$, goes to zero linearly as well. On the other hand, when the gradients are perturbed by random additive noise, as we shall discuss in the next section, $\E[f(x_k)-f^*]$ does {\it not} go to zero.
\subsection{Performance metrics under gradient noise: Rate and robustness}\label{rate_robustness_quad}
In this section, we use the dynamical system representation of the algorithms given in \eqref{dyn_sys_quad2:main} to study the limiting behavior of the expected suboptimality $\E[f(x_k)-f^*]$. 
We show that this sequence converges, i.e., the limit of the expected suboptimality exists and it is equal to the limit superior in \eqref{eq:robust_def}.
We also provide the associated convergence rate,
and present an explicit characterization of the limiting value using insights from robust control theory. More specifically, consider the shifted state sequence $\{\tx_k\}_k$ 
generated according to \eqref{dyn_sys_quad2:main}. Since $w_k$ is zero mean for all $k\geq 0$ (Assumption \ref{asmp1}), by taking the expectation of \eqref{dyn_sys_quad2:main}, we obtain  
\begin{align}\label{iterates_rate}
\E[\tx_{k}] &= A_Q^k~\tx_{0},\quad \forall\ k\geq 0.
\end{align}
Therefore, under the assumption that $\rho(A_Q) <1$, the sequence $\{\E[\tx_{k}]\}_k$ converges to zero with an asymptotic linear rate $\rho(A_Q)$. Note that the state sequence $\{\xi_k\}_k$ and the iterates $\{x_k\}_k$ can be related by defining $T = I_d$ for GD and $T=[I_d \quad 0_d]$ for AG 
so that $x_k = T \xi_k$ for all $k\geq 0$. 

Recall the robustness definition given in~\eqref{eq:robust_def}. In the next lemma, we focus on the 
suboptimality sequence, $\{f(x_k)-f^*\}_k$ for quadratic $f$ and 
we show that the limit, 
\begin{equation}\label{h_2_stochastic}
\mathcal{J} = \frac{1}{\sigma^2}\lim_{k \to \infty} \E[f(x_k)-f^*],
\end{equation}
exists; 
moreover, for some 
$\{\varepsilon_k\}_k\subset[0,\infty)$ such that $\lim_{k\to\infty} \varepsilon_k = 0$, we have 
\begin{equation}\label{quad_func}
\big|\E[f(x_k)-f^*]-\sigma^2 \mathcal{J}\big| \leq \psi_0~ \big(\rho(A_Q)+\varepsilon_k\big)^{2k},\quad \forall\ k\geq 0,
\end{equation}
where 
$\rho(A_Q)$ is the spectral radius of $A_Q$, and $\psi_0$ is a constant that may depend on the initialization $x_0$. This shows that the sequence $\{\E[f(x_k)-f^*]\}_k$ converges to an interval around the origin with radius $\sigma^2 \mathcal{J}$, and the convergence is linear with an asymptotical rate that is arbitrarily close to $\rho(A_Q)^2$. 
It is therefore natural to define the normalized radius, $\mathcal{J}$, as \emph{robustness} of the system to gradient noise, i.e., if this radius is bigger, it means that the asymptotic error of the algorithm in terms of the function value is larger; hence, 
the algorithm is less robust to the injected noise.

The limit in \eqref{h_2_stochastic} can be evaluated by using the tools from \emph{standard $H_2$ theory} arising in robust control of dynamical systems (see e.g. \cite{haddad1994parameter}) as we shall explain below. The $H_2$-norm is a well-known fundamental metric for quantifying the robustness of a linear dynamical system to noise in control engineering and has been widely used in designing the parameters of control systems subject to noise. Given arbitrary matrices $(A,B,C)$ and $D=0_d$, consider a linear system as in \eqref{dyn_sys:main} but 
{\it without} feedback $\phi$. Suppose there exists $\xi^*$ and $y^*$ such that $\xi^*=A\xi^*$ and $y^*=C\xi^*$.
The $H_2$-norm of this linear system, 
denoted by $H_2(A,B,C)$, measures the stationary variance of the output response $\{y_k\}$ to unit white noise input \cite{zhou1996robust}, i.e.,
{
   \beq H_2^2(A,B,C) \triangleq \lim_{k \to \infty} \frac{1}{\sigma^2}\E \|y_k-y^*\|^2. \label{def-H2-norm}
   \eeq
}%
The $H_2$ norm admits alternative definitions, which are all equivalent for linear systems (see e.g. \cite{zhou1996robust,Stoorvogel93}). When it is clear from the context, we will remove the dependency of the $H_2$ norm to the system matrices $(A,B,C)$. 
The $H_2$-norm can be computed as 
\begin{equation}\label{formula-h2_1}
H_2^2(A,B,C) = \Tr(C X_0 C^\top)
\end{equation}
where $X_0$ solves the \emph{discrete Lyapunov} equation\footnote{The value of $H_2^2$ can also be computed as
$\Tr(B^\top \tilde{X_0} B)$ where $\tilde{X_0}$ solves $A^\top \tilde{X_0} A - \tilde{X_0} + C^\top C = 0$.}:
\begin{equation}\label{eqn-lyapunov}
A X_0 A^\top - X_0 + B B^\top = 0
\end{equation}
(see e.g. \cite{haddad1994parameter,zhou1996robust}). Moreover, if $BB^\top$ is positive definite and $A$ is \textit{discrete-time stable} (i.e., $\rho(A)<1$), the solution admits the following formula:
\begin{equation}\label{soln-Lyapunov-eqn}
X_0 = \sum_{k=0}^\infty (A^\top)^k B^\top B  A^k
\end{equation}
(see e.g. \cite{zhou1996robust}). We will show in the following lemma that the limit $\mathcal{J}$ in \eqref{h_2_stochastic} exists for quadratic objectives. Our proof technique is based on relating $\mathcal{J}$ to the $H_2$ norm of a transformed linear system as follows: We first rewrite the suboptimality $f(x_k) - f^*$ in terms of the iterates $x_k$: 
\begin{align}\label{quad_func_err_3}
f(x_k) - f^* 
= \tfrac{1}{2} (x_k-x^*)^\top Q (x_k-x^*) = \tfrac{1}{2} (T\tx_k)^\top Q (T\tx_k) 
= (R T \tx_k)^\top (R T \tx_k),
\end{align}
 where we used the fact that 
 $x_k = T \xi_k$, 
 and $\frac{1}{2} Q = R^\top R$ is the Cholesky decomposition of $\frac{1}{2}Q$. 
If we consider the system defined by matrices $(A_Q, B, RT)$, it follows from the definition of the $H_2$ norm \eqref{def-H2-norm} and \eqref{quad_func_err_3} that 
    \beq \mathcal{J} = H_2^2(A_Q,B,\tiR) \quad \mbox{where} \quad \tiR \triangleq RT. \eeq
By \eqref{formula-h2_1} and \eqref{eqn-lyapunov}, we have $\mathcal{J} = \Tr(\tiR X \tiR^\top)$,
where $X$ is the solution to 
\begin{equation}\label{lyapunov_AQ}
A_Q X A_Q^\top - X + B B^\top = 0.
\end{equation}
\begin{lemma} \label{H_2_quad_equi}
Consider the linear dynamical system \eqref{dyn_sys_quad2:main} defined by the matrices $(A_Q, B, C)$. If 
$\rho(A_Q)<1$, 
then the limit in \eqref{h_2_stochastic} exists, i.e., $\sigma^2\mathcal{J} = \lim_{k \to \infty} \E[f(x_k)-f^*]$, and there exists a non-negative sequence $\{\varepsilon_k\}_k$ such that $\lim_k \varepsilon_k = 0$ and
\begin{eqnarray*}
\left|\E[f(x_k)-f^*]-\sigma^2 \mathcal{J}\right| \leq~\psi_0~\left(\rho(A_Q) +\varepsilon_k\right)^{2k}, \quad\forall\ k\geq 0,
\end{eqnarray*}
holds for some explicitly given positive constant $\psi_0$ that depends on the initialization $x_0$. Furthermore, when $A_Q$ is symmetric, 
$\varepsilon_k = 0$ for every $k\geq 0$.
\end{lemma}
\begin{proof}
Using \eqref{quad_func_err_3} and 
$\mathcal{J} = \Tr(\tiR X \tiR^\top)$, we obtain
\begin{equation}
\begin{split}
\E[f(x_k) - f^*] - &\sigma^2 H_2^2(A_Q, B, \tilde{R}) = \E[(\tilde{R} \tx_k)^\top (\tilde{R} \tx_k)]- \sigma^2 \Tr(\tilde{R} X \tilde{R}^\top)\\
&= \Tr(\tilde{R} \E[\tx_k \tx_k^\top] \tilde{R}^\top) - \sigma^2 \Tr(\tilde{R} X \tilde{R}^\top) = \Tr(\tilde{R}(V_k - \sigma^2 X)\tilde{R}^\top)
\end{split}
\end{equation}
where 
$V_{k} \triangleq \E[\xi_k\xi_k^\top]$ for 
$k\geq 0$. 
It follows from \eqref{dyn_sys_quad2:main} that 
\begin{equation} \label{recursive_var}
\begin{split}
&V_k = \E[\tx_k\tx_k^\top] = \E[(A_Q \tx_{k-1} + B w_{k-1})(A_Q \tx_{k-1} + B w_{k-1})^\top] \\
&= A_Q V_{k-1} A_Q^\top + \sigma^2 BB^\top
\end{split}
\end{equation}
holds for all $k\geq 1$, where in the last equality we used the fact that the random vector $w_{k-1}$ is zero-mean, independent of $\xi_{k-1}$, and has covariance matrix $\E[w_{k-1} w_{k-1}^T] = \sigma^2\id$. Moreover, by \eqref{lyapunov_AQ} we have
$X = A_Q X A_Q^T + BB^T$; hence, 
subtracting $\sigma^2 X$ from both sides of \eqref{recursive_var},
we obtain
\begin{equation}\label{rec_Vk}
V_k - \sigma^2 X = A_Q (V_{k-1}-\sigma^2 X) A_Q^\top = A_Q^k (V_0-\sigma^2 X) (A_Q^\top)^k
\end{equation}
where the last equality comes from recursively using the first equality. This implies
{
\begin{align}
|\Tr(\tilde{R}(V_k - \sigma^2 X)\tilde{R}^\top)| &= |\Tr(\tiR A_Q^k (V_0-\sigma^2 X) (A_Q^\top)^k \tiR^\top)|\nonumber\\
&=|\Tr((V_0-\sigma^2 X)(\tiR A_Q^k)^\top(\tiR A_Q^k))|\nonumber\\
& \leq m \|V_0-\sigma^2 X\| \|\tiR A_Q^k\|^2 \leq m \|V_0-\sigma^2 X\| \|\tiR\|^2 \|A_Q^k\|^2, \label{error-bound}
\end{align}}%
where $\|.\|$ is the 
spectral norm, and 
the first inequality in \eqref{error-bound} follows from the Von Neumann's trace inequality which states that for any two $m \times m$ matrices $U$ and $V$ with singular values $||U\|_2=u_1 \geq ... \geq u_m$ and $||V\|_2=v_1 \geq ... \geq v_m$, respectively, we have$|\Tr(UV)| \leq \sum_{i=1}^m u_i v_i$.
Finally, it follows from the Gelfand's formula that there exists a sequence of non-negative numbers $\{\varepsilon_k\}_k$ such that for every $k \geq 0$,
     $\|A_Q^k\|_2 \leq \left(\rho(A_Q) + \varepsilon_k\right)^k$ and $\lim_k \varepsilon_k = 0$. Note that when $A_Q$ is symmetric, we have $\| A_Q^k\|_2 = \rho(A_Q)^k$ so that we can choose $\varepsilon_k = 0$. Inserting this bound into \eqref{error-bound}, we obtain the desired result.
\end{proof}
It is worth noting that for a strongly convex quadratic function in the form of $f(x) = \tfrac{1}{2} x^\top Q x - p^\top x + r$, a similar line of argument as in Lemma \ref{H_2_quad_equi} shows that 
$\mathcal{J}'$ in \eqref{eq:robust_def_alternative} is in fact equal to $\lim_{k\to \infty}\frac{1}{\sigma^2}\E[\norm{x_k-x^*}^2]=H_2^2(A_Q,B,T)$.
Next we focus on the GD and AG algorithms, discuss the dependence of their convergence rate and robustness on the parameters (stepsize $\alpha$ and momentum $\beta$) and show how to formulate an optimization problem that systematically trades off convergence rate and robustness.
\subsection{Gradient descent (GD) method}
The dynamical system representation of GD, choosing the $A, B, C$ as in
\eqref{gradient_ABCD} yields 
	$$ \E[\xi_{k+1}] = A_Q \E[\xi_k] , \quad \mbox{with} \quad A_Q = \id - \alpha Q.$$  
As shown in Lemma \ref{H_2_quad_equi}, the convergence rate of GD is given by $\rho(A_Q)^2$. For GD, we will suppress the dependence of $\rho(A_Q)$ on $A_Q$ and use the notation $\rho(\alpha)$ to highlight the effect of the stepsize $\al$. Since $A_Q$ is symmetric, $\rho(\alpha)$ can be computed as
\begin{equation}\label{eq-rate}
\rho(\al) = \rho(A_Q) = \|A_Q\| = \max\{|1-\alpha \mu|, |1 - \alpha L|\}.
\end{equation}
$\alpha \in (0,2/L)$ is a necessary condition for global linear convergence; otherwise, $\rho(\al) \geq 1$. In particular, it is well-known that the fastest rate is achieved for the stepsize
{
  \beq \bar{\alpha} \triangleq \arg \min_{\alpha \geq 0} \rho(A_Q) = \frac{2}{\mu+L},
  	\label{eqn-fast-stepsize}
  \eeq
}%
which leads to a convergence rate of $\bar{\rho} = 1 - \frac{2}{\kappa + 1}$. The choice of the stepsize not only affects the 
rate (see~\eqref{eq-rate}) 
but also the robustness of the GD algorithm to gradient noise. The following proposition provides an analytical characterization of the robustness $\mathcal{J}$ of the GD method as a function of the stepsize, which we denote by $\mathcal{J}(\alpha)$ to highlight its dependence on $\alpha$. 

\begin{proposition} \label{quad_GD_thm}
Let $f$ be a quadratic function of the form $f(x) = \tfrac{1}{2} x^\top Q x - p^\top x + r$. Consider the GD iterations given by \eqref{gradient_update} with constant stepsize $\alpha \in (0,2/L)$. Then the robustness of the GD method is given by
{
\begin{align}
\mathcal{J}(\alpha) &=  \sum_{i=1}^d\frac{\al^2 \lambda_i }{2(1-(1-\alpha \lambda_i)^2)} = \al \sum_{i=1}^d\frac{1}{2(2-\al \lambda_i)} \label{GD_quad_main:b},
\end{align}}%
where $0<\mu = \lambda_1 \leq \lambda_2 \leq ... \lambda_d=L$ are the eigenvalues of $Q$.
\end{proposition}
\begin{proof}
We first show that without loss of generality we can assume $Q$ is a diagonal matrix. Let $Q = U\Lambda U^\top$ be the eigenvalue decomposition of $Q$ where $U$ is a unitary matrix and $\Lambda = \diag(\lambda_1,...,\lambda_d)$ is a diagonal matrix containing the eigenvalues of $Q$. 
Multiplying $A_Q$ by $U^\top$ and $U$ from left and right leads to
\begin{equation} \label{UTAQU}
U^\top A_Q U = U^\top (\id - \alpha Q) U = \id - \alpha \Lambda = A_\Lambda,
\end{equation}
where $A_\Lambda \triangleq \id - \alpha \Lambda$ is a diagonal matrix. Similarly, we multiply the Lyapunov equation \eqref{lyapunov_AQ} from left and right by $U^\top$ and $U$, which yields
$U^\top A_Q X A_Q^\top U - U^\top X U + \alpha^2\id = 0$,
where we have used the fact that $B=-\alpha\id$ for the dynamical system representation of the GD method (see \eqref{gradient_ABCD}). It follows from \eqref{UTAQU} that $A_Q = U(\id - \alpha \Lambda) U^\top$, which when plugged into the Lyapunov equation above, yields
$(\id - \alpha \Lambda) U^\top X U (\id - \alpha \Lambda) - U^\top X U + \alpha^2 \id = 0$.
This means that the matrix $U^\top X U$ solves the Lyapunov equation obtained by replacing $A_Q$ by $A_\Lambda$ in \eqref{lyapunov_AQ}. Furthermore, the Cholesky decomposition of $\frac{1}{2}\Lambda$ is equal to $\sqrt{\frac{1}{2}}\Lambda^{1/2}$; thus, the robustness $\mathcal{J}$, 
corresponding to $H_2^2(A_\Lambda, B, \sqrt{\frac{1}{2}}\Lambda^{1/2}T)$, 
is equal to
{
\begin{equation*}
\frac{1}{2}\Tr(\Lambda^{1/2}T U^\top X U (\Lambda^{1/2}T)^\top) = \frac{1}{2}\Tr(\Lambda^{1/2}U^\top X U\Lambda^{1/2}) = \Tr(\tiR X \tiR^\top),
\end{equation*}}%
where we used $T= \id$ for GD for the first equality and the fact that the Cholesky decomposition of $\frac{1}{2}Q$ is $(\sqrt{\frac{1}{2}} \Lambda^{1/2} U^\top)^\top (\sqrt{\frac{1}{2}} \Lambda^{1/2} U^\top)$ to obtain the second equality. Therefore, robustness $\mathcal{J}$ would be invariant if we were to replace $Q$ by $\Lambda$ and solve the Lyapunov equation \eqref{lyapunov_AQ} for $(A_\Lambda, B)$ instead of $(A_Q,B)$. With this replacement, it is easy to verify that the solution of the Lyapunov equation is 
$X_\Lambda = \diag\left(\frac{\alpha^2}{1-(1-\alpha \lambda_1)^2},...,\frac{\alpha^2}{1-(1-\alpha \lambda_d)^2})\right)$
as $A_\Lambda$ and $B$ are both diagonal. Plugging this solution into $\frac{1}{2} \Tr(\Lambda^{1/2} X_\Lambda \Lambda^{1/2})$ implies
$\mathcal{J}(\al) = \sum_{i=1}^d\frac{\al^2 \lambda_i }{2(1-(1-\alpha \lambda_i)^2)}= \al \sum_{i=1}^d\frac{1}{2(2-\al \lambda_i)}$
which completes the proof.
\end{proof}
\begin{rmk}
Proposition \ref{quad_GD_thm} also shows that the robustness $\mathcal{J}(\al)$ for the GD method is an increasing function of $\alpha$. This means choosing a smaller stepsize leads to GD being more robust which has been previously observed in the literature for both additive and multiplicative deterministic noise \cite{lessard2016analysis,fazlyab2017analysis}.
\end{rmk}
Having explicit expressions for both convergence rate and robustness for GD (see \eqref{eq-rate} and \eqref{GD_quad_main:b}), given an allowable deviation $\epsilon>0$ from the optimal convergence rate $\bar{\rho}=1 - \frac{2}{\kappa + 1}$, a natural approach to account for the trade-off between these two measures is to choose the stepsize $\alpha$ that results in the most robust algorithm satisfying the rate constraints, i.e., optimizing
{
\begin{equation}
\label{eq:H2-minimization}
 \min_{\alpha \in (0,2/L)} \mathcal{J}(\alpha)  \quad \mbox{subject to} \quad \rho(\alpha) \leq (1+\epsilon)\bar{\rho}.
\end{equation}}%
This problem is equivalent to the following convex problem for $\epsilon\in[0, \frac{2}{\kappa-1})$ (which ensures that the upper bound on the rate is less than one and the optimization problem \eqref{eq:H2-minimization} admits a solution):
{
\begin{equation}
\label{eq:H2-aux-problem}
\min_{\alpha \in (0,2/L)} \mathcal{J}(\alpha)  \quad \mbox{subject to} \quad \frac{1}{1-\rho^2(\alpha)} \leq \frac{1}{1-(1+\epsilon)^2\bar{\rho}^2}.
\end{equation}}%
Indeed, $1/(1-\rho^2)$ is a nondecreasing convex function for $\rho\in(0,1)$ and $\rho(\alpha)$ is convex in $\alpha$; therefore, both $1/(1-\rho(\alpha)^2)$ and  $\mathcal{J}(\alpha)$ in \eqref{GD_quad_main:b} are convex for $\alpha\in(0,\frac{2}{L})$ and is increasing in $\alpha$. Moreover, \eqref{eq:H2-aux-problem} satisfies the Slater condition. Thus, strong duality implies that there exists $\tau$ (which is a function of $\epsilon$) such that the above minimization problem is equivalent to the following unconstrained problem:
{
\begin{equation} \label{opt-pbm-gd}
\alpha_*(\tau)\triangleq \argmin_{\alpha \in (0,2/L)} F_\tau(\alpha) \triangleq \mathcal{J}(\alpha) + \tau \frac{1}{1-\rho^2(\alpha)}.
\end{equation}}%
The parameter $\tau>0$ determines the trade-off between rate and 
robustness. For small~$\tau$, the dominant term in the cost would be $\mathcal{J}(\alpha)$ so that we expect the optimal stepsize to be small since $\mathcal{J}(\al)$ is an increasing function of $\al$. On the other hand, for large enough $\tau$, the convergence rate is the dominant term in the cost; therefore, one would expect the optimal stepsize (that solves the problem \eqref{opt-pbm-gd}) to be close to $\bar{\alpha}$ which corresponds to the fastest achievable rate $\bar{\rho}$ (see \eqref{eqn-fast-stepsize}). In order to get more intuition about the effect of the choice of the stepsize parameter, we next give an illustrative example in dimension $d=2$ to show the behavior of the optimal $\alpha_*(\tau)$ as the tradeoff parameter $\tau$ is varied from zero to infinity. For computational tractability, we consider the unconstrained version of the problem given in \eqref{opt-pbm-gd}.\footnote{In Proposition \ref{prop-gd-optimality} of the appendix, we derive the first-order conditions for $\alpha_*(\tau)$ that allows it to be computed up to an arbitrary accuracy.}
\begin{ex}\label{example}
In dimension $d=2$, let $\tau=2$ and consider the parameters
{
   \beq\mu = \lambda_1 = 0.1 \quad \mbox{and} \quad L=\lambda_2 = 1 \quad \mbox{with} \quad \kappa =\frac{L}{\mu}=10. \label{params-mu-L}
   \eeq
}%
The first-order optimality conditions for \eqref{opt-pbm-gd} is derived in Proposition \ref{prop-gd-optimality} which is equivalent to a polynomial root finding problem in $\alpha$ for a polynomial of degree~$4$. The roots of polynomials can be found up to arbitrary accuracy by calculating the eigenvalues of the corresponding companion matrix \cite{edelman1995polynomial}, for instance using the \verb+roots+ function in {\sc Matlab}. After a careful examination of all the roots, we conclude that the optimal stepsize $\alpha_*$ that minimizes the cost $F_\tau(\alpha)$ is $\al_* \approx 1.5055$ which gives the rate $\rho(\alpha_*)\approx 0.8494$ and robustness $\mathcal{J}(\alpha_*) \approx 1.9294$. This point is marked on Figure~\ref{fig-GD} below which shows the robustness level as a function of the optimal convergence rate~$\rho$ when we change $\tau$ from zero (corresponds to the rightmost point in the curve) to infinity (corresponds to the uppermost point in the curve) for the parameters in \eqref{params-mu-L}. 
\begin{figure}[ht!]
  \centering
    \begin{minipage}[t]{0.3\textwidth}
    \includegraphics[width=\textwidth]{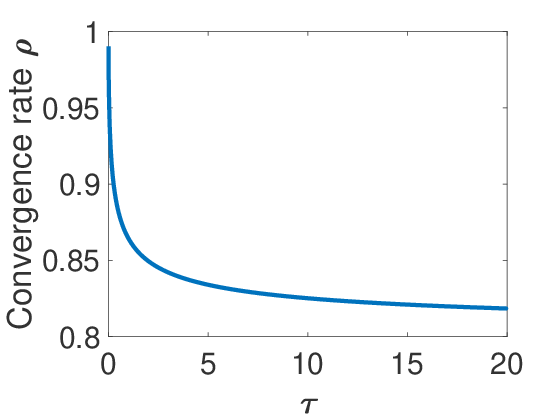}
  \end{minipage}
  \hfill
  \begin{minipage}[t]{0.3\textwidth}
    \includegraphics[width=\textwidth]{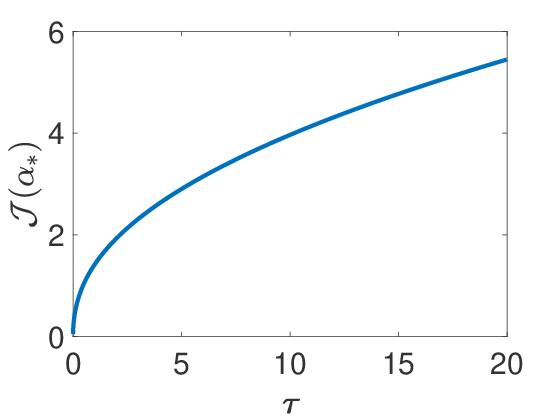}
  \end{minipage}
  \hfill
  \begin{minipage}[t]{0.3\textwidth}
    \includegraphics[width=\textwidth]{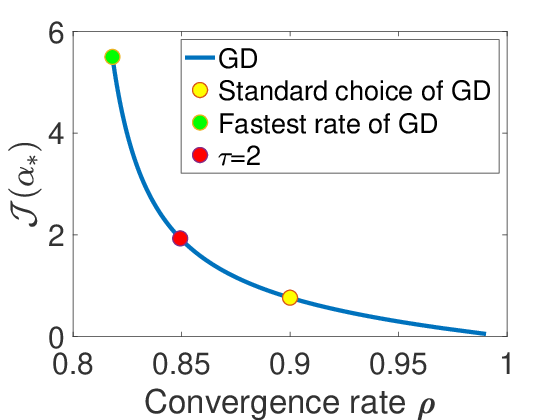}
  \end{minipage}
  \caption{\label{fig-GD} Left and Middle: The behaviors of the convergence rate $\rho$ and robustness $\mathcal{J}$ computed at the optimal stepsize $\al_*$, as a function of the trade-off parameter $\tau$. Right: The robustness level as a function of convergence rate again as $\tau$ varies from zero to infinity.}
\end{figure}

The left and the middle panels of Figure \ref{fig-GD} show the convergence rate and robustness corresponding to the optimal stepsize $\al_*$ as a function of the trade-off parameter~$\tau$. 
As $\tau$ goes to $0$, the robustness term is more dominant which requires a smaller stepsize; therefore, $\alpha_*$ goes to $0$ and $\rho(\alpha_*)$ thus goes to $1$. As $\tau$ becomes larger, convergence rate becomes more important, and the stepsize also becomes larger to ensure faster convergence. In particular, as $\tau$ goes to infinity, $\al_*$ goes to $\bar{\alpha}$ given in \eqref{eqn-fast-stepsize}
leading to the fastest rate, $\bar{\rho} = \frac{\kappa - 1}{\kappa + 1} \approx 0.8182$.

Finally, the rightmost panel of Figure \ref{fig-GD} illustrates the trade-off between the rate and robustness. We se that for small $\tau$, the optimal stepsize $\alpha_*$ is smaller which 
implies improved robustness but slower convergence.  As $\tau$ grows, we achieve faster rate at the expense of being less robust to the additive gradient noise. In addition, the points corresponding to the fastest rate, i.e., $\al=2/(\mu+L)$, and standard parameter choice $\al=1/L$ for GD has been marked on this trade-off curve.
\end{ex}
We see from Figure \ref{fig-GD} that smaller values of $\rho$ (or equivalently smaller values of $\frac{1}{1-\rho^2}$) are accompanied by larger values of $\mathcal{J}$. This suggests that the product $\mathcal{J} \frac{1}{1-\rho^2}$ cannot be too small for any choice of the stepsize $\alpha$. The next lemma shows that 
there are some fundamental limits (lower bounds) 
on how robust the GD can be.
\begin{proposition} \label{quad_GD_imposs}
Let $\rho(\al)$ and $\mathcal{J}(\al)$ be given by \eqref{eq-rate} and \eqref{GD_quad_main:b}, respectively. Then, the following inequality holds
$\mathcal{J}(\alpha) \geq \big(1-\rho^2(\alpha)\big)\sum_{i=1}^d \frac{1}{8\lambda_i}$ 
for any choice of the stepsize $\alpha > 0$. 
\end{proposition}
\begin{proof}
It follows from \eqref{eq-rate} that for every $i\in \{1,...,d\}$, we have $\rho(\al) \geq |1-\alpha \lambda_i|$. This implies that
$\tfrac{1}{1-\rho(\al)^2} \geq \tfrac{1}{1-(1-\alpha \lambda_i)^2}$.
Multiplying both sides by $\frac{\alpha^2 \lambda_i}{2(1-(1-\alpha \lambda_i)^2)}$ and summing over all $i$ yields
$\tfrac{1}{1-\rho^2} \sum_{i=1}^d \tfrac{\alpha^2 \lambda_i}{2(1-(1-\alpha \lambda_i)^2)} \geq \sum_{i=1}^d \tfrac{\alpha^2 \lambda_i}{2(1-(1-\alpha \lambda_i)^2)^2}$.
Given the explicit characterization of $\mathcal{J}(\al)$ in Proposition \ref{quad_GD_thm} (see \eqref{GD_quad_main:b}) we obtain
{
\begin{equation} \label{imposs_bound_1}
\frac{1}{1-\rho^2} \mathcal{J}(\al) \geq \frac{1}{2} \sum_{i=1}^d \frac{\alpha^2 \lambda_i}{(1-(1-\alpha \lambda_i)^2)^2}.
\end{equation}}%
The right hand side of \eqref{imposs_bound_1} admits a lower bound as follows:
{
\begin{equation} \label{imposs_bound_2}
\frac{1}{2} \sum_{i=1}^d \frac{\alpha^2 \lambda_i}{(1-(1-\alpha \lambda_i)^2)^2} = \frac{1}{2} \sum_{i=1}^d \frac{\alpha^2 \lambda_i}{(\alpha \lambda_i (2-\alpha \lambda_i))^2} = \frac{1}{2} \sum_{i=1}^d \frac{1}{\lambda_i (2-\alpha \lambda_i)^2} \geq \sum_{i=1}^d \frac{1}{8\lambda_i},
\end{equation}}%
where the last inequality follows from the fact that $|2-\alpha \lambda_i| \leq 2$. Using the lower bound \eqref{imposs_bound_2} along with \eqref{imposs_bound_1} completes the proof.
\end{proof}
\subsection{Accelerated gradient (AG) method}
The dynamical system representation of AG, given $A,B,C$ in \eqref{nest_ABCD} leads to
{
\begin{equation} \label{A_Q_nest}
A_Q = \begin{bmatrix} 
	(1+\be)(\id-\alpha Q) & -\be(\id - \alpha Q) \\ 
    \id 	  &  0_d
\end{bmatrix}.
\end{equation}}%
We will first formulate an analogous problem to \eqref{opt-pbm-gd} for the AG method to design the parameters $(\alpha,\beta)$ in a way to find a trade-off between the rate and the robustness. Because AG has the pair $(\alpha,\beta)$ as design parameters, the analogue of \eqref{opt-pbm-gd} is
{
\begin{equation} \label{opt-pbm-ag}
\left(\alpha_*, \beta_*\right) \triangleq \arg\min_{(\al,\beta) \in \mathcal{S}} F_\tau(\alpha,\beta) \triangleq \mathcal{J}(\alpha,\beta) + \tau \frac{1}{1-\rho(\alpha,\beta)^2}
\end{equation}}%
where $\mathcal{J}(\alpha,\beta)$ is the robustness to the noise for the system 
\eqref{dyn_sys_quad2:main}, $\rho(\alpha,\beta)$ is the convergence rate of AG with parameters $(\alpha,\beta)$ and $\mathcal{S}$ is the set of all possible choices of the tuple $(\alpha,\beta)$ so that the AG iterations are globally convergent, i.e., 
\begin{equation}\label{stability_def}
\mathcal{S} = \left\{ (\al,\be) ~:~ \rho(A_Q) < 1, \al\geq 0, \be\geq 0 \right\} \subset \R^2.
\end{equation}
We call 
$\mathcal{S}$, the \emph{stability region} of $AG$, in analogy with the stability region of numerical methods that arise in the discretization of continuous-time differential equations. 

We next provide an explicit characterization for the convergence rate and robustness of AG for any given parameters $(\alpha,\beta) \in \mathcal{S}$. The convergence rate $\rho$ of the AG method as a function of $\alpha$ and $\beta$ is well-known. Diagonalizing the $A_Q$ matrix using the eigenvalue decomposition of $Q$, it can be shown after some computations that the rate $\rho=\rho(\alpha,\beta)$ admits the following formula
\begin{equation} \label{rho_nest}
\rho(\alpha,\beta) = \rho(A_Q) = \max\{\rho_\mu(\alpha,\beta), \rho_L(\alpha,\beta)\}
\end{equation}
where $A_Q$ is defined by \eqref{A_Q_nest} and $\rho_\lambda$ is defined for $\lambda\in\{\mu, L\}$ as follows:
{
\begin{equation}\label{def-delta-lambda}
\rho_\lambda(\alpha,\beta) = \begin{cases}
	\frac{1}{2}|(1+\be)(1-\al\lambda)| + \frac{1}{2}\sqrt{\Delta_\lambda} & \mbox{if } \Delta_\lambda \geq 0 \\
    \sqrt{\be(1-\al\lambda)} & \mbox{otherwise} 
\end{cases},\quad \Delta_\lambda = (1+\be)^2 (1-\al\lambda)^2 - 4\be(1-\al\lambda)
\end{equation}}%
(see e.g. \cite[Appendix A]{lessard2016analysis}, \cite[Section 4.3]{candes-restart-acc-grad}). The explicit expression \eqref{rho_nest} for the rate allows us to characterize the set $\mathcal{S}$ in the next proposition whose proof can be found in the appendix. We illustrate the set $\calS$ in Figure \ref{fig-stability-23} for different choices of the parameters $\mu$ and $L$.\footnote{We note that the stability region of a second-order difference equation that arises in accelerated algorithms that are sublinearly convergent for weakly convex quadratic functions has been studied in \cite{flammarion2015averaging}, however these results do not apply to the set $\mathcal{S}$ as we do not require the rate to be accelerated (we consider not only accelerated rates but also any rate $\rho$ less than one) and we consider strongly convex functions instead of weakly convex functions.}
\begin{figure}[ht!]
  \centering
\begin{minipage}[b]{0.36\textwidth}
    \includegraphics[width=\textwidth]{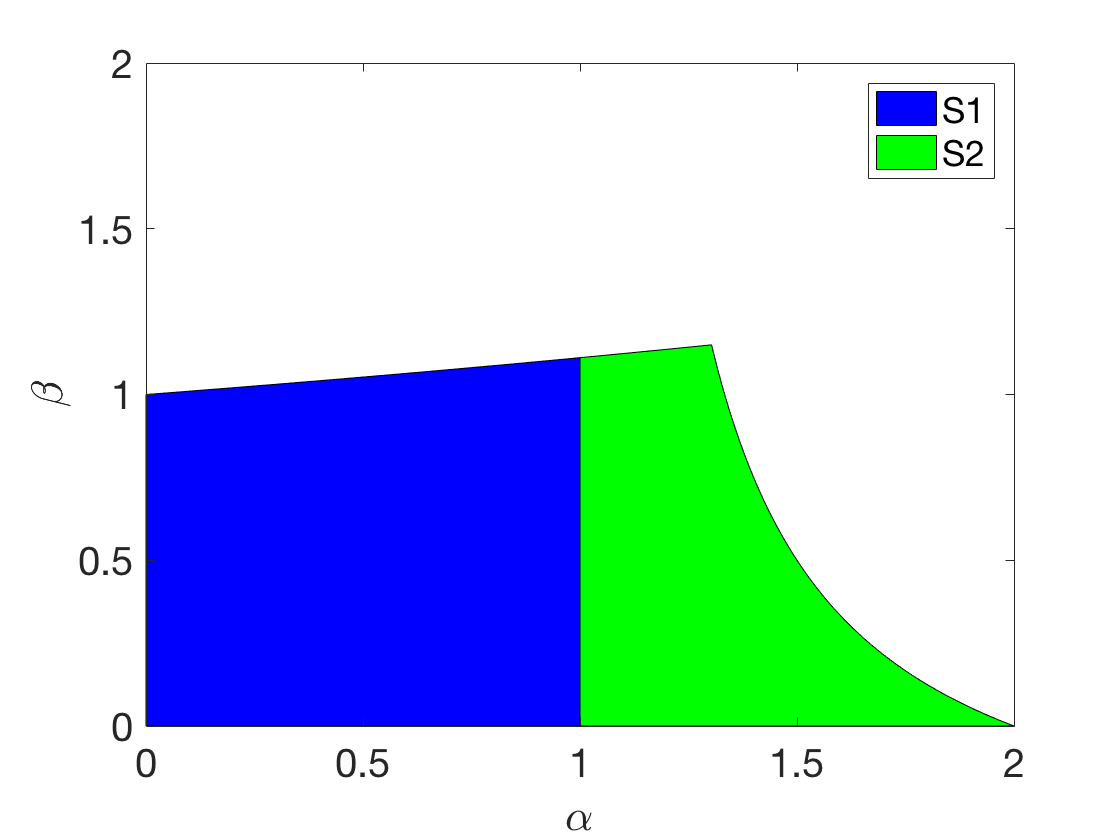}
\end{minipage}
\qquad\qquad
\begin{minipage}[b]{0.36\textwidth}
 \includegraphics[width=\textwidth]{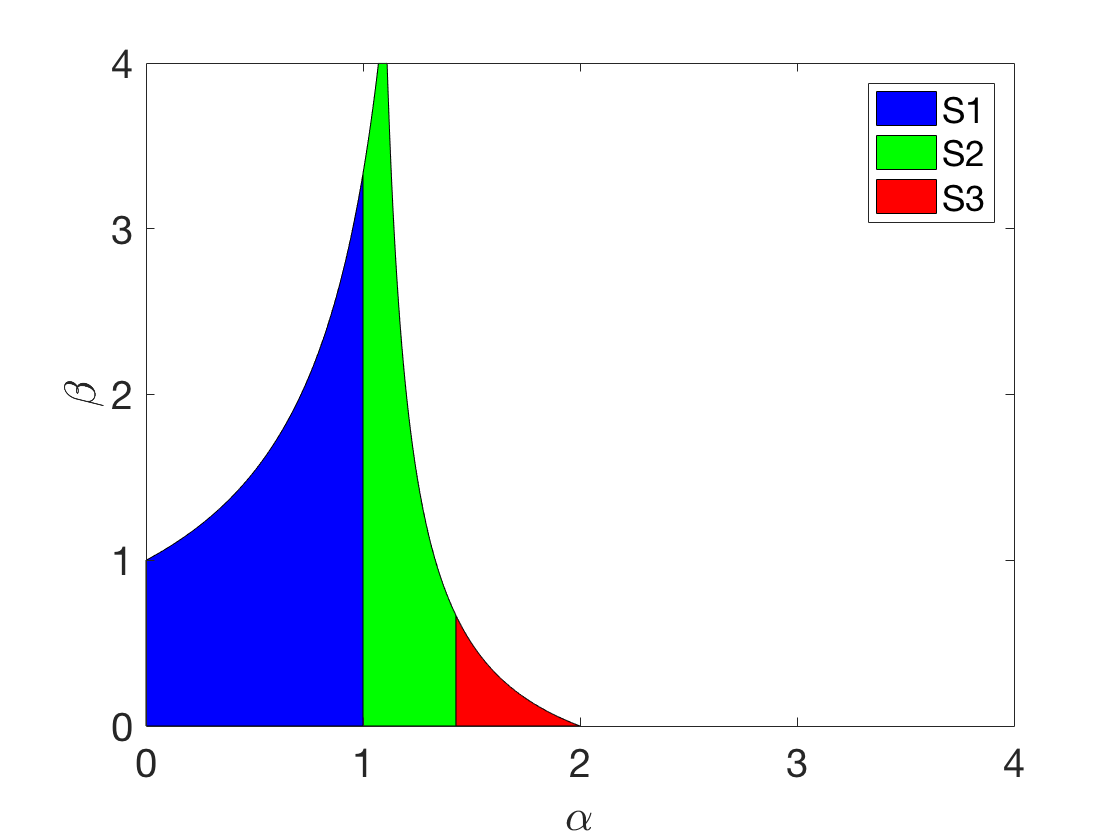}
  \end{minipage}
  \caption{\label{fig-stability-23} Left: The stability region $\calS = \calS_1 \cup\calS_2\cup\calS_3 $ with parameters $\mu=0.7$ and $L=1$. Right: The stability region $\calS = \calS_1 \cup\calS_2$ with parameters $\mu=0.1$ and $L=1$.}
\end{figure}
\begin{proposition}\label{prop-stability-set-ag}
Let $\mathcal{S}$ be the stability set of Nesterov's accelerated method defined by \eqref{stability_def}. Then its closure is given by the union of the following three sets:
{
\begin{equation} \label{nest_stability}
\begin{split}
& \mathcal{S}_1 := \left\{(\alpha, \beta): 0 \leq \alpha \leq \frac{1}{L},\  0 \leq \beta(1-\alpha \mu) \leq 1 \right\}, \\
& \mathcal{S}_2 :=  \left\{(\alpha, \beta): \frac{1}{L} < \alpha \leq \min\left\{\frac{2}{L}, \frac{1}{\mu}\right\}, \alpha L-1 \leq \frac{1}{2\beta+1}, \beta(1-\alpha \mu) \leq 1 \right\},\\
& \mathcal{S}_3 :=   \left\{(\alpha, \beta): \frac{1}{\mu} \leq \alpha \leq \frac{2}{L}, \alpha L-1 \leq \frac{1}{2\beta+1} \right\},\\
\end{split}
\end{equation}
}
with the convention that $\mathcal{S}_3$ is the empty set if $\mu < \frac{L}{2}$.
\end{proposition}

The 
next proposition gives a characterization of the robustness $\mathcal{J}(\al,\be)$ of AG whose proof can be found in the Appendix \ref{sec-appendix-h2-nesterov}.
\begin{proposition}\label{prop-h2-nesterov}
Let $f$ be a quadratic function of the form $f(x) = \tfrac{1}{2} x^\top Q x - p^\top x + r$. Consider the AG iterations given by \eqref{nest:main} with parameters $(\alpha,\beta) \in \mathcal{S}$ . Then the robustness of the AG method is given by
{
\begin{align}\label{eqn-h2sq-acc-grad}
	\mathcal{J}(\alpha,\beta) = \sum_{i=1}^d u_{\alpha,\beta}(\lambda_i)
\end{align}}%
where $\mu = \lambda_1 \leq \lambda_2 \leq \dots \leq \lambda_d = L$ are the eigenvalues of $Q$ and
	{
    \beq \uab (\lambda) \triangleq \al \frac{1+\be(1-\al \lambda)}{2(1-\be(1-\al\lambda))(2+2\be-\al\lambda(1+2\be))} \label{eq-s-lam}.
    \eeq
    }
    In the special case, choosing $\beta=0$ reduces to the formula \eqref{GD_quad_main:b} derived for GD.
\end{proposition}
Since we have an exact characterization of $\mathcal{J}(\al, \be)$, we can derive the optimality conditions for the problem \eqref{opt-pbm-ag} by an approach similar to  Proposition~\ref{prop-gd-optimality}, where the optimizer can be characterized as a root of some polynomial. In dimension $d=2$, given parameters $\mu$ and $L$, the optimizer is easy to compute. However, in high dimensions, this is computationally expensive as it would require determining all the eigenvalues of $Q$ which can be as expensive as optimizing the objective function $f$. Nevertheless, exploiting the convexity properties of the function $u_{\al,\be}(\lambda)$, we develop a tractable upper bound for $\mathcal{J}(\al, \be)$ that only depends on $\mu$ and $L$, hence tractable. Moreover, in the numerical experiments section, we present experiments illustrating that this approach can lead to good performance in terms of trading the speed and the robustness of an algorithm.

To develop this upper bound, first, we show in Lemma \ref{lem-H2-cvx} that the function $\uab(\lambda)$ defined in \eqref{eq-s-lam} is convex in $\lambda \in [\mu,L]$ for fixed $(\alpha,\beta) \in \calS$. Therefore, its maximum is attained at one of the endpoints of this interval, i.e.,
	$$\uab(\lambda) \leq  \buab \triangleq \max \left[ \uab(\mu), \uab(L)\right] \quad \mbox{for} \quad \lambda \in [\mu,L].$$
Substituting this upper bound in \eqref{eqn-h2sq-acc-grad} and \eqref{opt-pbm-ag} leads to 
\begin{align}\label{eqn-h2sq-acc-grad-ub}
	\mathcal{J}(\al, \be) \leq \bar{\mathcal{J}}(\alpha,\beta) \triangleq d  \buab 
\end{align}
and the relaxed optimization problem is
{
\begin{equation} \label{opt-pbm-ag-ub}
\left(\alpha_*, \beta_*\right) \triangleq \arg\min_{(\al,\beta) \in \mathcal{S}} \bar{F}_\tau(\alpha,\beta) \triangleq \bar{\mathcal{J}}(\alpha,\beta) + \tau \frac{1}{1-\rho(\alpha,\beta)^2}.
\end{equation}}%
This objective only depends on $\mu$ and $L$ and is differentiable everywhere in the interior of the stability region $\calS$ except when the first term is not differentiable, i.e., when $\uab(\mu) = \uab(L)$, or the second term is not differentiable, i.e., when $\rho_\mu = \rho_L$ or $\Delta_\mu = 0$ or $\Delta_L=0$. Furthermore, following a similar approach as in Example \ref{example}, the first order optimality conditions with respect to $\alpha$ and $\beta$ results in low-order polynomials (that are independent of the dimension $d$) which can be solved efficiently up to any accuracy. Thus, other than checking the non-differentiable points of $\bar{F}_\tau$, 
the bottleneck in computational complexity 
is determined by computing the roots of a polynomial with a small degree (whose degree is independent from the dimension $d$), which is easy to compute even in high dimensions.

\section{Strongly Convex Functions}\label{strong_smooth_case}
In this section, we generalize our analysis to smooth strongly convex functions 
$f\in S_{\mu,L}(\mathbb{R}^d)$.  
Note that we can rewrite~\eqref{noisy-dyn_sys:main} as
\begin{align}
\xi_{k+1} = A \xi_k + B (\nabla f(y_k) + w_k),\quad  
y_k  = C \xi_k, 
\label{noisy-dyn_sys2:main}
\end{align}
where $w_k \in \R^d$ models the additive noise and satisfies Assumption \ref{asmp1}. Similar to the previous section, we define $\ty_k \triangleq y_k-x^*$ and $\tx_k \triangleq \xi_k - \xi^*$ where $\xi^*$ is equal to $x^*$ for GD and $[x^{*^\top} x^{*^\top}]^\top$ for AG, and in both cases we have $\xi^* = A\xi^*$ and $x^* = C \xi^*$ where $A$ and $C$ are given in \eqref{gradient_ABCD} and \eqref{nest_ABCD} for GD and AG, respectively. We 
use the equation $x_k = T \xi_k$ to relate $x_k$ and $\xi_k$ for any $k \geq 0$, where $T=I_d$ for GD and $T=[I_d \quad 0_d]$ for AG. 
To simplify the notation, we define $\bar{f}:\R^m\rightarrow\R$ such that $\baf(\xi)=f(T\xi)$ for all $\xi\in\R^m$ which means $\baf(\xi_k) = f(x_k)$ for all $k \geq 0$.
\subsection{Rate and robustness}
The goal 
is to extend the definitions of rate and robustness from the quadratic case to 
general strongly convex functions. 
We will use
{
\begin{equation} \label{convex_h_2_stochastic}
\mathcal{J} = \limsup\limits_{k\to\infty} \frac{1}{\sigma^2}\E[f(x_{k}) - f^*]
\end{equation}}%
(provided also in \eqref{eq:robust_def}) to define the robustness of an algorithm and study the convergence rate of the expected suboptimality to an interval around zero with radius $\sigma^2 \mathcal{J}$. 
For both GD and AG, our main results provide upper bounds of the 
form: 
\begin{equation}\label{main_contraction}
\E[f(x_{k}) - f^*] \leq \rho^{2k} \psi_0 + \sigma^2 R,\quad \forall\ k \geq 0,
\end{equation}
where $\psi_0$, $R$, and $0<\rho<1$ are non-negative numbers all of which depend on algorithm parameters and the initial point $x_0$. 
Clearly, $R$ is an upper bound on $\mathcal{J}$; we will show 
in this section that {our bounds are tight}. Moreover, we also recover the fastest known rates in the literature in the absence of noise ($\sigma$=0).
%
Our upper bounds only depend on $\mu$ and $L$, and are computationally tractable and explicit in some cases. With these upper bounds, one can formulate an optimization problem 
similar to that of the previous section to find the algorithm parameters that can achieve a particular trade-off between rate and robustness.
\subsection{Rate and robustness trade-off analysis using Lyapunov functions}\label{sec:lyapunov-convergence}
We use a \emph{Lyapunov function} approach to provide a bound as in \eqref{main_contraction} for both GD and AG methods. In particular, 
we consider a family of Lyapunov functions parameterized by a non-negative constant $c$ and a positive semidefinite matrix $P$ as
\begin{equation}
V_{P,c}(\xi) \triangleq V_P(\xi) +c(\bar{f}(\xi)-f^*),
\end{equation}
where $V_P(\xi) \triangleq (\xi-\xi^*)^\top P (\xi-\xi^*)$, and study the change in the Lyapunov function $V_{P,c}(\xi)$ along 
$\{\xi_k\}_k$ generated by the dynamical system representation \eqref{noisy-dyn_sys2:main}.  

Our first 
result shows how for some positive semidefinite matrix $P$ and $c=0$, $V_p(\xi)\triangleq V_{P,0}(\xi)$ 
evolves along the iterations and provide a characterization of the difference $\E[V_P(\xi_{k+1})]- \rho^2 \E[V_P(\xi_k)]$ for any $k \geq 0$ and $\rho \geq 0$.
\begin{lemma} \label{storage_update}
Consider the Lyapunov function $V_P(\xi)= (\xi-\xi^*)^\top P  (\xi-\xi^*)$ where $P\succeq 0$. Then, we have
{
\begin{equation}\label{storage_update_Expec}
\E[V_P(\xi_{k+1})] = \E\left[\begin{bmatrix} 
	\xi_k-\xi^*\\
    \nabla f(y_k)
\end{bmatrix}^\top
\begin{bmatrix} 
	A^\top P A & A^\top P B\\
    B^\top P A & B^\top P B
\end{bmatrix}
\begin{bmatrix} 
	\xi_k-\xi^*\\
    \nabla f(y_k)
\end{bmatrix}\right]+\sigma^2 \Tr(B^\top P B).
\end{equation}}%
\end{lemma}
\begin{proof}
Since $\xi^*=A\xi^*$, 
\eqref{noisy-dyn_sys2:main} implies $\tx_{k+1} = A \tx_k + B (\nabla f(y_k) + w_k)$ for $k\geq 0$. Therefore,
{
\begin{align}
\E[V_P(\xi_{k+1})] &= \E[\tx_{k+1}^\top P \tx_{k+1}] = \E[(A \tx_k + B (\nabla f(y_k) + w_k))^\top P (A \tx_k + B (\nabla f(y_k) + w_k))] \nonumber \\
&= \E\left[\begin{bmatrix} 
	\tx_k\\
    \nabla f(y_k)
\end{bmatrix}^\top
\begin{bmatrix} 
	A^\top P A & A^\top P B\\
    B^\top P A & B^\top P B
\end{bmatrix}
\begin{bmatrix} 
	\tx_k\\
    \nabla f(y_k)
\end{bmatrix}\right]+\sigma^2 \Tr(B^\top P B). \label{lyapunov_iterates4}
\end{align}}%
where 
we use the fact that $w_k$ is zero mean and independent from $\tx_k$ and $y_k$ 
 and the assumption that $\mathbb{E}(w_k w_k^T) = \sigma^2 \id$ 
to derive \eqref{lyapunov_iterates4}.

\end{proof}
\begin{corollary}\label{storage_update_2}
For any $\rho\in(0,1)$ and $P \in \mathbb{S}^m_+$, 
$V_P(\xi)$ satisfies 
{
\begin{equation}\label{storage_contraction}
\E[V_P(\xi_{k+1})]- \rho^2 \E[V_P(\xi_k)]
=\E\left[\begin{bmatrix} 
	\xi_k-\xi^*\\
    \nabla f(y_k)
\end{bmatrix}^\top
\Phi(A,B,P,\rho)
\begin{bmatrix} 
	\xi_k-\xi^*\\
    \nabla f(y_k)
\end{bmatrix}\right]+\sigma^2 \Tr(B^\top P B)
\end{equation}}%
where
{
\begin{equation}
\Phi(A,B,P,\rho) \triangleq
\begin{bmatrix} 
	A^\top P A - \rho^2 P & A^\top P B\\
    B^\top P A & B^\top P B
\end{bmatrix}.\label{Phi}
\end{equation}}%
\end{corollary}
We next show how to provide upper bounds on the improvement in $\E[V_{P,c}(\xi_k)]$ by imposing {\it Matrix Inequalities} (MIs).
In particular, given $A$, $B$, and $C$ defining the first-order algorithm, we assume there exists a symmetric matrix $X\in\mathbb{S}^{m+d}$ such that
\begin{equation}\label{eq:LMI}
X \succeq \Phi(A, B, P, \rho)    
\end{equation}
for some $P\in\mathbb{S}^m_+$ and $\rho\in(0,1)$. Moreover, we assume that for some non-negative constants $\Gamma$ and $c$, the same $\rho$ and $X$ satisfy
{
\begin{align}
\label{eq:energy_dissipation}
    \E\left[
\begin{bmatrix} 
	\xi_k-\xi^*\\
    \nabla f(y_k)
\end{bmatrix}^\top
X
\begin{bmatrix} 
	\xi_k-\xi^*\\
    \nabla f(y_k)
\end{bmatrix}
\right] \leq c(\rho^2 \E[\bar{f}(\xi_k)-f^*] - \E[\bar{f}(\xi_{k+1})-f^*]+\sigma^2 \Gamma)
\end{align}}%
for every $k \geq 0$; hence, it follows from \eqref{eq:LMI} and \eqref{eq:energy_dissipation} along with \eqref{storage_contraction} that 
\begin{align*}
    \rho^2 \E[V_{P,c}(\xi_k)]+\sigma^2 (\Tr(B^\top P B)+ c\Gamma) \geq \E[V_{P,c}(\xi_{k+1})],\quad \forall~k\geq 0.
\end{align*}
Thus, for all $k\geq 0$, we have
{
\begin{align}
\label{eq:general_result}
    \E[V_{P,c}(\xi_k)]\leq \rho^{2k} V_{P,c}(\xi_0)+\frac{1-\rho^{2k}}{1-\rho^2} \sigma^2 R_P,\quad R_P\triangleq \Tr(B^\top P B)+c\Gamma. 
\end{align}}%
This MI based approach has been used in the literature to study the convergence rate of first-order methods, e.g., \cite{lessard2016analysis,hu2017dissipativity}. Here we use it to characterize their rate and robustness under additive gradient noise.
\subsection{Gradient descent (GD) method for strongly convex functions}
Recall the GD update rule
\begin{equation} \label{GD_general}
x_{k+1} = x_k - \alpha(\nabla f(x_k) + w_k),
\end{equation}
which admits the dynamical system representation in \eqref{noisy-dyn_sys2:main} with $A, B, C$ as in \eqref{gradient_ABCD}.
The 
next theorem extends the result of Proposition \ref{quad_GD_thm} to general strongly convex functions and characterize the behavior of 
$\{x_k\}_k$ under additive gradient error.
\begin{proposition} \label{general_gd_cont}Let \Sml, and consider the GD iterations given by \eqref{GD_general}. Assume there exist $\rho\in(0,1)$ and $p>0$ such that 
\begin{equation}\label{LMI_1}
X_0 \succeq \Phi(A,B,P,\rho)
\end{equation}
holds where
{
$X_0 = 
\begin{bmatrix} 
	2\mu L I_d & -(\mu+L)I_d\\
    -(\mu+L)I_d & 2I_d
\end{bmatrix}$} and $P=p I_d$.
Then
for all $k\geq 0$:
\begin{equation}
\label{eq:GD_result}
\E[\norm{x_{k}-x^*}^2] \leq \rho^{2k} \norm{x_0-x^*}^2 + \frac{1-\rho^{2k}}{1-\rho^2}~\sigma^2\alpha^2 d.
\end{equation}
\end{proposition}
\begin{proof}
Noting that $\xi_k=y_k$ for GD, it follows from \eqref{GD_X} with $x=\xi_k$ and $y=x^*$ that \eqref{eq:energy_dissipation} holds for $X=X_0$ and $c=0$.
Moreover, \eqref{LMI_1} implies that \eqref{eq:LMI} holds for $X=X_0$ and $P=p\otimes I_d$; therefore, \eqref{eq:general_result} yields
{
\begin{equation}
p \E[\norm{x_{k}-x^*}^2] \leq \rho^{2k} p \norm{x_0-x^*}^2 + \frac{1-\rho^{2k}}{1-\rho^2} \sigma^2 \Tr(B^\top PB),\quad \forall\ k\geq 0.
\end{equation}}%
With $B=-\alpha I_d$ for GD, we have $\Tr(B^\top P B)=\alpha^2 p d$ and this completes the proof.
\end{proof}
Note that for a fixed $\al$, a smaller $\rho$ makes both terms of \eqref{eq:GD_result} smaller as $\tfrac{1-\rho^{2k}}{1-\rho^2}$ is an increasing function of $\rho$. If $\alpha \in (0,2/L)$, it was shown in \cite{lessard2016analysis} that there exist $(p, \rho)$ such that the MI in~\eqref{LMI_1} holds; moreover, for a given $\alpha$ fixed, the smallest $\rho\in (0,1)$ for which such a positive $p$ exists is equal to
\begin{equation}
\label{eq:rho_GD}
\rho_{GD}(\al) = \max\{|1-\alpha \mu|, |1 - \alpha L|\},
\end{equation}
as in \eqref{eq-rate} given for quadratic functions. Using 
$\rho = \rho_{GD}(\al)$ in \eqref{eq:GD_result} leads to the following 
upper bound for GD.\footnote{\label{J_GD_iterates}{
Trivially, $\mathcal{J}' \leq \frac{\alpha^2 d}{1-\rho_{GD}(\al)^2}$. More details are provided in Appendix \ref{iterate_case} as a supplementary material.}}
\begin{corollary} \label{gd_main_bounds}
Let \Sml. 
Consider the GD iterations given by \eqref{GD_general} with constant stepsize $\alpha \in (0,2/L)$. Then, for all $k\geq 0$,
\begin{equation*}
\E[f(x_k)-f^*] \leq \rho_{GD}(\al)^{2k}~\psi_0+\Big(1-\rho_{GD}(\al)^{2k}\Big)~\sigma^2 R_{GD}(\alpha),
\end{equation*}
where $\psi_0=\frac{L}{2}\norm{x_0-x^*}^2$ and 
$\rho_{GD}(\alpha)$ is given in \eqref{eq:rho_GD}. As a consequence,
\begin{equation}\label{R_GD}
\mathcal{J} \leq R_{GD}(\alpha),\quad\mbox{where}\quad R_{GD}(\alpha) \triangleq \frac{L\alpha^2 d}{2(1-\rho^2_{GD}(\al))}.   
\end{equation}
\end{corollary}
\begin{proof}
Using the fact that $f(x_k)-f(x^*) \leq \frac{L}{2}\|x_k-x^*\|^2$ for $k\geq 0$ together with Proposition \ref{general_gd_cont} yields the desired result.
\end{proof}
{Note that by substituting $\rho_{GD}$ in \eqref{R_GD}, we obtain $R_{GD} = \bigO(\alpha d)$. This bound is tight, as Proposition \ref{quad_GD_thm} implies that for quadratic functions $\mathcal{J} = \Theta (\alpha d)$.}
\subsection{Accelerated Gradient (AG) method for strongly convex functions}
\label{sec:AG_general}
We next consider the AG algorithm with gradient noise given by
\begin{subequations} \label{AG_general}
\begin{align}
x_{k+1} &= y_k - \alpha (\nabla f(y_k)+w_k) \label{AG_general:a}\\
y_k &= (1+\beta)x_k - \beta x_{k-1}. \label{AG_general:b}
\end{align}
\end{subequations}
As before, these iterations admit the dynamical system representation in \eqref{noisy-dyn_sys2:main} with $A, B, C$ as in \eqref{nest_ABCD}.
We use the following result which extends Lemma 3 in \cite{hu2017dissipativity} to the case with noisy gradient.
\begin{lemma} \label{stronglyconvex_nest_lem}
Let \Sml, and consider the dynamical system representation of AG with $\xi_k=[x_k^\top,x_{k-1}^\top]^\top$.
Then, for any $\rho\in(0,1)$,
{
\begin{align*}
\begin{bmatrix} 
	\xi_k-\xi^*\\
    \nabla f(y_k)
\end{bmatrix}^\top
\left(\rho^2 X_1 + (1-\rho^2)X_2\right)
\begin{bmatrix} 
	\xi_k-\xi^*\\
    \nabla f(y_k)
\end{bmatrix}
\leq &\rho^2 (f(x_k)-f^*)-(f(x_{k+1})-f^*) \\
& + \frac{L\alpha^2}{2}\|w_k\|^2 -\alpha(1-L\alpha)\nabla f(y_k)^\top w_k
\end{align*}}%
holds for all $k\geq 0$, 
where $X_1 = \tilde{X}_1 \otimes I_d$ and $X_2 = \tilde{X}_2 \otimes I_d$ with
{
\begin{align*}
\tilde{X}_1 = \frac{1}{2}
\begin{bmatrix} 
	\be^2\mu& -\be^2\mu & -\be \\
    -\be^2\mu & \be^2\mu & \be  \\
    -\be  & \be  & \alpha(2-L\alpha) 
\end{bmatrix},~ 
\tilde{X}_2 = \frac{1}{2}
\begin{bmatrix} 
	(1+\be)^2\mu & -\be(1+\be)\mu & -(1+\be) \\
    -\be(1+\be)\mu & \be^2\mu & \be  \\
    -(1+\be)  & \be  & \alpha(2-L\alpha) 
\end{bmatrix}.
\end{align*}}%
\end{lemma}
\begin{proof}
Setting $x=x_k$ and $y=y_k$ in the second inequality in \eqref{ineq_S} 
leads to
\begin{equation} \label{ineq_S1}
f(x_k)-f(y_k) \geq \nabla f(y_k)^\top (x_k-y_k) +\frac{\mu}{2} \Vert x_k-y_k \Vert^2. 
\end{equation}
Similarly, setting $x=x_{k+1}=y_k-\alpha \nabla f(y_k) - \alpha w_k$ and $y=y_k$ in \eqref{ineq_S} 
yields to
{
\begin{equation}\label{ineq_S2}
\begin{split}
f(y_k)-f(x_{k+1}) &\geq \nabla f(y_k)^\top (\alpha \nabla f(y_k) + \alpha w_k) -\frac{L\alpha^2}{2} \Vert \nabla f(y_k) + w_k \Vert^2\\
& = \frac{\alpha}{2}(2-L\alpha) \Vert \nabla f(y_k) \Vert^2 - \frac{L\alpha^2}{2} \Vert w_k \Vert^2  + \alpha(1-L\alpha) \nabla f(y_k)^\top w_k.
\end{split}
\end{equation}}%
Summing up \eqref{ineq_S1} and \eqref{ineq_S2} implies
\begin{equation} \label{ineq_S3}
\begin{split}
f(x_k)-f(x_{k+1}) \geq & \frac{1}{2}
\begin{bmatrix} 
	x_k-y_k\\
    \nabla f(y_k)
\end{bmatrix}^\top
\begin{bmatrix} 
	\mu\id & \id \\
    \id & \alpha(2-L\alpha)\id
\end{bmatrix}
\begin{bmatrix} 
	x_k-y_k\\
    \nabla f(y_k)
\end{bmatrix}\\
& -\frac{L\alpha^2}{2} \Vert w_k \Vert^2 + \alpha(1-L\alpha)\nabla f(y_k)^\top w_k.
\end{split}
\end{equation}
Note that $x_k - y_k = x_k - ((1+\be) x_k - \be x_{k-1}) = \be (x_{k-1}-x_k)$; hence, \eqref{ineq_S3} implies
\begin{equation}
\label{eq:X1_ineq}
{f(x_k)-f(x_{k+1}) +\frac{L\alpha^2}{2} \Vert w_k \Vert^2 - \alpha(1-L\alpha)\nabla f(y_k)^\top w_k \geq
\begin{bmatrix} 
	x_{k-1}-x^*\\
    x_k-x^*\\
    \nabla f(y_k)
\end{bmatrix}^\top
X_1
\begin{bmatrix} 
	x_{k-1}-x^*\\
    x_k-x^*\\
    \nabla f(y_k)
\end{bmatrix}}.
\end{equation}
Next, in a similar way, setting $x=x^*$ and $y=y_k$ in 
\eqref{ineq_S}, and summing the 
second inequality with \eqref{ineq_S2} leads to 
\begin{equation}
\label{eq:X2_ineq}
{f(x^*)-f(x_{k+1}) + \frac{L\alpha^2}{2}\|w_k\|^2 -\alpha(1-L\alpha)\nabla f(y_k)^\top w_k\geq
\begin{bmatrix} 
	x_{k-1}-x^*\\
	x_{k}-x^*\\
    \nabla f(y_k)
\end{bmatrix}^\top
X_2
\begin{bmatrix} 
	x_{k-1}-x^*\\
	x_k-x^*\\
    \nabla f(y_k)
\end{bmatrix}}.
\end{equation}
Multiplying \eqref{eq:X1_ineq} by $\rho^2$ and \eqref{eq:X2_ineq} by $1-\rho^2$, and summing them will lead to the desired result.
\end{proof}
\begin{proposition} \label{AG_mainresult}
Let \Sml, 
and consider the AG iterations given by \eqref{AG_general}. Assume there exist $\rho\in(0,1)$, $P \in \mathbb{S}_{+}^{2d}$, and $c_0,c\geq 0$ such that
\begin{equation}\label{LMI_nest}
c_0 X_0
+cX(\rho)
\succeq \Phi(A,B,P,\rho),\quad \hbox{where}
\end{equation}
{
\begin{equation*}
X_0 = 
\begin{bmatrix} 
	2\mu L~C^\top C & -(\mu+L)C^\top\\
    -(\mu+L)C & 2I_d
\end{bmatrix}, \quad X(\rho) = \rho^2 X_1 + (1-\rho^2)X_2
\end{equation*}}%
for $X_1$ and $X_2$ defined in Lemma \ref{stronglyconvex_nest_lem}. Then the following bounds hold for all $k\geq 0$:
{
\begin{equation}\label{AG_stochastic_bound1}
\E[V_{P,c}(\xi_k)] \leq \rho^{2k} 
V_{P,c}(\xi_0) + \frac{1-\rho^{2k}}{1-\rho^2}\sigma^2\alpha^2\left(\frac{c}{2}Ld+\Tr(P_{11})\right),
\end{equation}}%
where 
$P_{11}\in\mathbb{S}_+^d$ is the submatrix of $P$ formed by its first $d$ rows and $d$ columns. 
\end{proposition}
\begin{proof}
Using \eqref{GD_X} for $x=y_k$ and $y=x^*$ along with the fact $y_k = C \xi_k$ yields
\begin{equation}
\label{eq:X0}
\begin{bmatrix} 
	\xi_k-\xi^*\\
    \nabla f(y_k)
\end{bmatrix}^\top
X_0
\begin{bmatrix} 
	\xi_k-\xi^*\\
    \nabla f(y_k)
\end{bmatrix}\leq 0.
\end{equation}
This inequality along with Lemma \ref{stronglyconvex_nest_lem} implies that \eqref{eq:energy_dissipation} holds for $X = c_0 X_0 + cX(\rho)$ 
and $\Gamma=\tfrac{1}{2}L\alpha^2 d$. Moreover, \eqref{LMI_nest} implies that \eqref{eq:LMI} holds for this $X$. Therefore, \eqref{eq:general_result} holds and 
$\Tr(B^\top P B) = \alpha^2 \Tr(P_{11})$ completes the proof. 
\end{proof}
As stated in the previous proposition, the MI in \eqref{LMI_nest} provides us with $(P, \rho)$ pairs 
and through \eqref{AG_stochastic_bound1} we obtain an upper bound on $\sup_k \E[V_{P,c}(\xi_k)]$, which leads to a bound $R$ on $\mathcal{J}$. However, solving this $2d\times 2d$ MI becomes intractable as $d$ increases. To keep this MI invariant of the dimension,
we restrict our attention to the case that $P$ is in the form of $\tilde{P}\otimes I_d$, where $\tilde{P}$ is a $2\times 2$ symmetric positive semidefinite matrix; hence, \eqref{LMI_nest} becomes a $3\times 3$ MI.
The following corollary shows the robustness bound when $P=\tilde{P}\otimes I_d$ for some $\tilde{P}\in\mathbb{S}_+^2$. 
\begin{corollary}\label{AG_mainresult_P2}
Let \Sml, and consider the AG iterations given by \eqref{AG_general} with parameters $\alpha$ and $\beta$. Assume there exist $\rho\in(0,1)$, $\tilde{P} \in \mathbb{S}_{+}^{2}$, $c_0\geq 0$, and $c> 0$ such that
$c_0 X_0+cX(\rho) \succeq \Phi(A,B,P,\rho)$
with $X_0$ defined in Theorem \ref{AG_mainresult}, $X_1, X_2$ defined in Lemma \ref{stronglyconvex_nest_lem}, and $P=\tilde{P}\otimes I_d$. Then for all $k \geq 0$,
{
\begin{align}
&\E[f(x_k)-f^*] \leq \rho^{2k}~\psi_0+\Big(1-\rho^{2k}\Big)\sigma^2 R_{AG}(\alpha,\beta),\\
&R_{AG}(\alpha,\beta) \triangleq
\begin{dcases}
\frac{L\alpha^2 d}{2(1-\rho^2)}~ \frac{cL+2\tilde{P}_{11}}{cL+2(\tilde{P}_{11}-\tilde{P}_{12}^2/\tilde{P}_{22})}, 
& \tilde{P}_{22} >0, \\
\frac{L\alpha^2 d}{2(1-\rho^2)}, & \tilde{P}_{22} = 0.
\end{dcases}
\label{R_AG}
\end{align}}%
where $\psi_0=\frac{1}{c}V_{P,c}(\xi_0)$.
As a consequence,
$\mathcal{J} \leq R_{AG}(\alpha,\beta)$.
\end{corollary}
\begin{proof}
Note that since $\tilde{P}\in\mathbb{S}^2_+$, we have $\tilde{P}_{22}\geq 0$. If $\tilde{P}_{22} > 0$, then using Schur complements, $\tilde{P}$ can be written as sum of two positive semidefinite matrices:
\begin{equation*}
\tilde{P} =
\begin{bmatrix}
\tilde{P}_{11}-\tilde{P}_{12}^2/\tilde{P}_{22}& 0\\
0 & 0
\end{bmatrix}
+
\begin{bmatrix}
\tilde{P}_{12}^2/\tilde{P}_{22} &  \tilde{P}_{12}\\
\tilde{P}_{12} & \tilde{P}_{22}
\end{bmatrix}.
\end{equation*}
Hence, $V_P(\xi_k) \geq (\tilde{P}_{11}-\tilde{P}_{12}^2/\tilde{P}_{22}) \|x_k - x^*\|^2$. On the other hand, if $\tilde{P}_{22} = 0$, then $\tilde{P}\in\mathbb{S}^2_+$ implies that $\tilde{P}_{12} = 0$ as well and we get  $V_P(\xi_k) \geq \tilde{P}_{11} \|x_k - x^*\|^2$. In either case, substituting the derived lower bounds on $V_P(\xi_k)$ in \eqref{AG_stochastic_bound1} and using the facts that $\frac{2}{L}(f(x_k)-f^*) \leq \|x_k - x^*\|^2$ and $\Tr(B^\top P B) = \Tr(B B^\top P) = \alpha^2 \tilde{P}_{11} d$ completes the proof.
\end{proof}
\begin{rmk}
Note that for $\alpha\in(0,\frac{2}{L})$, setting $\be=0$ in AG yields GD algorithm with stepsize $\alpha$. Selecting $c_0=1$, $c=0$ and 
$\tilde{P} =
\begin{bmatrix}
\tilde{p} &0\\
0 & 0
\end{bmatrix}$, we observe that $\tilde{p}=L^2$ satisfies \eqref{LMI_nest}; therefore, Corollary~\ref{AG_mainresult_P2} implies that $\mathcal{J}\leq \frac{L\alpha^2 d}{2(1-\rho_{GD}(\alpha)^2)}$ for GD; hence, the result in~\eqref{R_GD} can be derived as a special case of Corollary~\ref{AG_mainresult_P2}
\end{rmk}
Using this result, the next corollary characterizes the rate and robustness of the AG method with 
a particular parameterization.
\begin{corollary} \label{agd_main_bounds_2}
Let \Sml, 
consider the AG iterations given by \eqref{AG_general} with constant stepsize $\alpha \in (0,1/L]$ and $\beta(\alpha)=\frac{1-\sqrt{\alpha\mu}}{1+\sqrt{\alpha\mu}}$. Then, for all $k\geq 0$,
{
\begin{equation*}
\E[f(x_k)-f^*] \leq \rho_{AG}(\al)^{2k}~\psi_0
 +\Big(1-\rho_{AG}(\al)^{2k}\Big)~\sigma^2 R_{AG}(\alpha)  
\end{equation*}}%
where $\psi_0= V_{P,1}(\xi_0)$, $\rho_{AG}(\alpha)\triangleq\sqrt{1-\sqrt{\alpha\mu}}$ and $R_{AG}(\alpha)\triangleq\frac{\alpha d}{2(1-\rho_{AG}(\alpha)^2)}(1+\alpha L)$; hence, 
$\mathcal{J} \leq R_{AG}(\alpha) {=\frac{\sqrt{\alpha}d}{2\sqrt{\mu}}(1+\alpha L)}$.
\end{corollary}
\begin{proof}
\cite[Theorem~2.3]{aybat2019universally} guarantees that for any $\alpha\in(0,1/L]$, the matrix inequality $X(\rho_{AG}(\alpha))
\succeq \Phi(A,B,P(\alpha),\rho_{AG}(\alpha))$ holds for $A$ and $B$ as in \eqref{nest_ABCD} corresponding to given $\alpha$ and $\beta(\alpha)=\frac{1-\sqrt{\alpha\mu}}{1+\sqrt{\alpha\mu}}$, where 
$P(\alpha)=\tilde{P}(\alpha)\otimes I_d$ for {
\begin{equation}\label{Nest_GeneralRate}
\tilde{P}(\alpha)\triangleq
\begin{bmatrix}
\sqrt{\frac{1}{2\alpha}}\\\sqrt{\frac{{\mu}}{2}}-\sqrt{\frac{1}{2\alpha}}    
\end{bmatrix}
\begin{bmatrix}
\sqrt{\frac{1}{2\alpha}} & \sqrt{\frac{{\mu}}{2}}-\sqrt{\frac{1}{2\alpha}}    
\end{bmatrix}.
\end{equation}}%
Therefore, the desired result follows from Corollary~\ref{AG_mainresult_P2}.
\end{proof}

It is worth noting that 
although solving 
$c_0 X_0+cX(\rho) \succeq \Phi(A,B,P,\rho)$ considering only lower-dimensional $P=\tilde{P}\otimes I_d$ is more restrictive, this small-dimensional MI can still recover the well-known 
rate $\bar{\rho}_{AG} = \sqrt{1-\sqrt{\frac{1}{\kappa}}}$ for the deterministic case, 
i.e., $\sigma=0$, in the literature \cite{nesterov_convex}
by setting the stepsize $\alpha=\frac{1}{L}$ and momentum parameter $\be(\alpha)=\frac{\sqrt{\kappa}-1}{\sqrt{\kappa}+1}$. As shown in \cite{hu2017dissipativity}, this claim can be verified by setting $P=\tilde{P}(\alpha)\otimes I_d$ with $\alpha=1/L$.
{In addition, 
for the case $L = \mu$}, substituting $\beta(\alpha)$ in the explicit expression of $\mathcal{J}$ in \eqref{eqn-h2sq-acc-grad} for quadratic functions, we obtain $\mathcal{J}(\alpha,\beta(\alpha))=$
$\Theta(\frac{\sqrt{\alpha}d}{\sqrt{\mu}})$ which implies that $R_{AG}$ is a tight bound for $\mathcal{J}$ in terms of $\alpha$ dependency.

It is worth noting that the best rate known in the literature for general \Sml~is $\rho_*=\sqrt{1-\sqrt{2\kappa-1}/\kappa}$ provided in \cite[Theorem 7]{safavi2018explicit}. However, this rate differs from $\bar{\rho}_{AG}$ just by a constant factor, i.e., $\frac{1-\rho_*^2}{1-\bar{\rho}_{AG}^2}\leq \sqrt{2}$. Moreover, for the special case of \Sml~being a quadratic function, the best linear rate for AG is $ 1-2/\sqrt{3\kappa+1}$  
(see~\cite[Prop 1]{lessard2016analysis} for an asymptotic analysis and \cite{bugra-wass-rate} for a non-asymptotic analysis). Therefore, we can conclude that $\rho=\mathcal{O}(\sqrt{1-1/\sqrt{\kappa}})$ and $\rho=\mathcal{O}(1-1/\sqrt{\kappa})$ denote the best known $\kappa$ dependency of the rate coefficient for general and quadratic \Sml, respectively. That said, since the focus of Section~\ref{strong_smooth_case} is on general strongly convex functions \Sml, in the following subsection in order to approximate the rate-robustness trade-off curves for AG, we consider $\bar{\rho}_{AG}=\sqrt{1-1/\sqrt{\kappa}}$ as the reference rate as it exhibits the optimal $\kappa$ dependency.

\subsection{
Approximating the rate and robustness trade-off curve}
Similar to \eqref{eq:H2-minimization}, \eqref{eq:H2-aux-problem}, and \eqref{opt-pbm-gd} in Section \ref{quad_case}, there are several ways of forming an optimization problem to trade-off the rate and robustness.
In this section we focus on strongly convex objectives \Sml, where unlike the quadratic objectives, we have access to upper bounds for the robustness measure rather than exact expressions. Therefore, we 
adopt a formulation similar to \eqref{eq:H2-minimization}
and vary $\epsilon$ to characterize the rate-robustness trade-off. In fact, the parameter $\epsilon$ shows how much we desire to lose in terms of convergence rate to gain robustness. 

In the rest, let $\bar{\rho}_{GD}\triangleq 1 - \frac{2}{\kappa + 1}$ and $\bar{\rho}_{AG}\triangleq\sqrt{1-\sqrt{1/\kappa}}$ denote the linear convergence rates for GD and AG from the literature for \Sml \cite{nesterov_convex}. In this section, we assume $\kappa\neq 1$, as the $\kappa=1$ case is trivial. We also let $\mathcal{J}_{{\rm GD},\epsilon}$ and $\mathcal{J}_{{\rm AG},\epsilon}$ be the best robustness value of GD and AG can achieve corresponding to rate $\rho_{{\rm GD}, \epsilon}\triangleq (1+\epsilon)\bar{\rho}_{GD}$ and $\rho_{{\rm AG}, \epsilon}\triangleq (1+\epsilon)\bar{\rho}_{AG}$, respectively. In the rest of this section, we discuss methodologies to derive tractable upper bounds on $\mathcal{J}_{{\rm GD},\epsilon}$ and $\mathcal{J}_{{\rm AG},\epsilon}$.

In particular, for GD, the best robustness level while asking for linear convergence with rate $\rho_{{\rm GD}, \epsilon}$ or faster is obtained by solving
\begin{equation}
\label{general_gd-problem}
\argmin_{\alpha \in (0,2/L)} R_{GD}(\alpha)  \quad \mbox{subject to} \quad \rho_{GD}(\alpha) \leq \rho_{{\rm GD}, \epsilon},
\end{equation}%
where $\rho_{GD}(\alpha)$ is given in \eqref{eq:rho_GD}. The function $\rho_{GD}(\alpha)$ is convex and piecewise linear in $\alpha$ over the interval $[0,2/L]$ with a unique minimum at $\bar{\rho}_{GD}$ and it satisfies $\rho_{GD}(0) = \rho_{GD}(2/L)=1$ on the boundary points. Therefore, it follows from this property that, given $\epsilon\in (0, \tfrac{2}{\kappa-1})$, there are exactly two $\alpha_\epsilon>0$ values such that $\rho_{GD}(\alpha_\epsilon)=\rho_{{\rm GD}, \epsilon}$ which we can explicitly compute as $\alpha_\epsilon=\frac{2-\epsilon(\kappa-1)}{L+\mu}$ or $\alpha_\epsilon=\frac{2 + \epsilon(\kappa-1)/\kappa}{L+\mu}$. The former value is strictly smaller as $\varepsilon>0$ and $\kappa>1$ here. From the formula \eqref{R_GD}, we have $R_{GD}(\alpha_\varepsilon) = \frac{L\alpha_\varepsilon^2 d}{2(1-\rho^2_{{\rm GD}, \epsilon})}$. Clearly one should select the smaller $\alpha_\epsilon$ value to minimize the robustness bound, i.e., a choice of $\alpha_\epsilon=\frac{2-\epsilon(\kappa-1)}{L+\mu}$ leads to $\rho_{{\rm GD}, \epsilon}$
rate with a robustness bound $R_{GD}(\alpha_\epsilon)$, i.e., $\mathcal{J}_{{\rm GD},\epsilon}\leq R_{GD}(\alpha_\epsilon)$.

For AG, we can also write an analogous optimization problem in order to trade rate with robustness:
\begin{align}\label{general_AG-problem}
\argmin_{\substack{\al, \be \geq 0, \tilde{P} \in \mathbb{S}_+^2\\ \rho, c_0, c \geq 0}}\{ R_{AG}(\alpha,\beta):\ \rho \leq \rho_{{\rm AG}, \epsilon},~c_0X_0+ {c}X(\rho) \succeq \Phi(A,B,\tilde{P} \otimes I_d,\rho)\} 
\end{align}%
with $X_0$, $X(\rho), R_{AG}$ defined in Corollary \ref{AG_mainresult_P2} {-- since we can scale $\tilde{P}$ with $c>0$, without loss of generality,} we restrict our attention to the case $c=1$ for treating the $c>0$ case.
This problem is in general non-convex and not easy to solve. Here, we consider two different ways to 
generate rate-robustness trade-off curves.

The first approach is similar to the one we used for GD.
In particular, consider Corollary~\ref{agd_main_bounds_2}, for $\alpha\in(0,1/L]$, choosing $\beta=\frac{1-\sqrt{\alpha\mu}}{1+\sqrt{\alpha \mu}}$ implies that $\rho_{AG}(\alpha)=\sqrt{1-\sqrt{\alpha\mu}}$. {We get $\rho_{AG}(\alpha_\epsilon)=\rho_{{\rm AG}, \epsilon}$ for} 
{
\begin{align}
\label{eq:alpha_eps}
\alpha_{\epsilon}=[1-\rho^2_{{\rm AG}, \epsilon}]^2/\mu=[1-(1+\epsilon)^2(1-\frac{1}{\sqrt{\kappa}})]^2/\mu, 
\end{align}}%
with $\epsilon\in \big[0, \sqrt{\frac{\sqrt{\kappa}}{\sqrt{\kappa}-1}}-1\big)$ to make sure the rate is smaller than $1$.
Thus, choosing $(\alpha,\beta)=(\alpha_\epsilon, \beta_\epsilon)$ with $\beta_\epsilon\triangleq\frac{1-\sqrt{\alpha_\epsilon\mu}}{1+\sqrt{\alpha_\epsilon \mu}}$ guarantees the rate $\rho_{{\rm AG}, \epsilon}$. In addition, Corollary~\ref{agd_main_bounds_2} implies the
robustness bound $R_{AG}(\alpha_\epsilon)=\mathcal{O}(\frac{\sqrt{\alpha_\epsilon}d}{\sqrt{\mu}})$ for this case.

For the second approach, 
we first grid the $(\alpha,\beta)\in(0,\tfrac{2}{L}]\times [0,1]$ parameter space, and use a numerical approach to find the best parameters for each $\epsilon$, {i.e., we solve a low-dimensional (in $\mathbb{R}^4$) convex semi-definite programming (SDP) problem for each possible $(\alpha,\beta)$ value from the grid, and we will pick the best one to determine the robustness bound. The SDPs arise from a convex approximation to the problem \eqref{general_AG-problem} as we now elaborate further.} First, we note that
$R_{AG}(\alpha,\beta)$ defined in \eqref{R_AG} is not convex in $\tilde{P}$; therefore, to obtain a tractable problem, we replace $R_{AG}(\alpha,\beta)$ with a convex upper bound. In particular, using Proposition~\ref{AG_mainresult}, it is straightforward to see
$$R_{AG}(\alpha,\beta) \leq \bar{R}_{AG}(\alpha,\beta)\triangleq\frac{\alpha^2 d}{2(1-\rho^2)}(L+2\tilde{P}_{11})$$
whenever there exists $\rho\in(0,1)$, $\tilde{P}\in\mathbb{S}_+^2$ and $\bar{c}\geq 0$ such that\footnote{Recall that we put $c=1$ in this section.}
\begin{align}
\label{eq:LMI_2}
\bar{c}X_0+X(\rho)\succeq\Phi(A,B,\tilde{P}\otimes I_d,\rho).
\end{align}
Note given $\epsilon\in[0, \sqrt{\frac{\sqrt{\kappa}}{\sqrt{\kappa}-1}}-1)$, setting $\rho=\rho_{{\rm AG}, \epsilon}$
within the matrix inequality in \eqref{eq:LMI_2}, we get a linear matrix inequality in $\bar{c}\geq 0$ and $\tilde{P}\in\mathbb{S}^2_+$ for fixed $(\alpha,\beta)$. Moreover, $\bar{R}_{AG}(\alpha,\beta)$ is linear in $\tilde{P}$. Hence, given the trade-off parameter $\epsilon>0$, we will approximately solve
{
\begin{align}\label{eq-rbar-eps}
\bar{\mathcal{R}}(\epsilon)\triangleq\min_{\al, \be,\bar{c} \geq 0, \tilde{P} \in \mathbb{S}_+^2}\{ \bar{R}_{AG}(\alpha,\beta):~\bar{c} X_0+ X({\rho_{{\rm AG}, \epsilon}}) \succeq \Phi(A,B,\tilde{P} \otimes I_d,\rho_{{\rm AG}, \epsilon})\} 
\end{align}}%
with $X_0$ and $X(\rho)$ defined in Corollary \ref{AG_mainresult_P2}, and $\bar{R}_{AG}$ as given above. In fact, for a fixed $(\alpha,\beta)$, this is a {small dimensional} convex SDP problem and can be solved easily.

Thus, we first grid the AG parameter space, i.e., $\{(\alpha_{i_1},\beta_{i_2})\}_{i_1\in\mathcal{I}_1,i_2\in\mathcal{I}_2}$ and for given trade-off parameter $\epsilon$, we solve $|\mathcal{I}_1| |\mathcal{I}_2|$ many 4-dimensional SDPs, i.e., for each $(i_1, i_2)\in\mathcal{I}_1\times \mathcal{I}_2$, 
\begin{align}\label{eq-rbar-grid}
\bar{\mathcal{R}}_{i_1,i_2}(\epsilon)\triangleq\min_{\bar{c}\geq 0, \tilde{P}\in\mathbb{S}^2_+} &\bar{R}_{AG}(\alpha_{i_1},\beta_{i_2})=\frac{\alpha_{i_1}^2 d}{2(1-\rho_{{\rm AG}, \epsilon}^2)}(L+2\tilde{P}_{11})\\
{\rm s.t.}& \quad \bar{c}X_0+X(\rho_{{\rm AG}, \epsilon})\succeq\Phi(A,B,\tilde{P}\otimes I_d,\rho_{{\rm AG}, \epsilon}).\nonumber
\end{align}
Clearly, $\mathcal{J}_{{\rm AG},\epsilon}\leq \bar{\mathcal{R}}(\epsilon)\leq \min_{i_1\in\mathcal{I}_1,i_2\in\mathcal{I}_2}\bar{\mathcal{R}}_{i_1,i_2}(\epsilon)$ where $\bar{\mathcal{R}}(\epsilon)$ is as in \eqref{eq-rbar-eps}. Note $(\alpha,\beta)=(\alpha_\epsilon, \beta_\epsilon)$, $\bar{c}=0$ and $\tilde{P}=\tilde{P}(\alpha_\epsilon)$ satisfies \eqref{eq:LMI_2} where $\alpha_\epsilon$ is given in \eqref{eq:alpha_eps},  $\beta_\epsilon=\frac{1-\sqrt{\alpha_\epsilon\mu}}{1+\sqrt{\alpha_\epsilon \mu}}$ and $\tilde{P}(\alpha)$ is defined in \eqref{Nest_GeneralRate}. Therefore, 
for any grid that contains $(\alpha_\epsilon, \beta_\epsilon)$ as one of the grid points, we have
\begin{align}
\label{eq:R_bound}
\bar{\mathcal{R}}(\epsilon)\leq \frac{\alpha_\epsilon^2 d}{2\left(1-\rho_{{\rm AG}, \epsilon}^2\right)}\left(L+\frac{1}{\alpha_\epsilon}\right)\leq \frac{\alpha_\epsilon d}{\left(1-\rho_{{\rm AG}, \epsilon}^2\right)}
=\frac{\sqrt{\alpha_\epsilon}d}{\sqrt{\mu}},
\end{align}
where the first inequality follows from $\tilde{P}_{11}=\tfrac{1}{2\alpha_\epsilon}$ for  $\tilde{P}=\tilde{P}(\alpha_\epsilon)$ and the second inequality follows from $\alpha_\epsilon\in(0,1/L]$. The equality above is a direct consequence of the identity \eqref{eq:alpha_eps}. Finally, according to the discussion at the end of the previous section,
for $\epsilon>0$ sufficiently close to $\sqrt{\frac{\sqrt{\kappa}}{\sqrt{\kappa}-1}}-1$, there exists a quadratic $f\in\mathbb{S}_{\mu,L}$ such that $\mathcal{J}(\alpha_\epsilon,\beta_\epsilon)=\Theta(\frac{\sqrt{\alpha_\epsilon}d}{\sqrt{\mu}})$ while \eqref{eq:R_bound} implies that $\mathcal{J}_{{\rm AG},\epsilon}\leq\bar{\mathcal{R}}(\epsilon)\leq \frac{\sqrt{\alpha_\epsilon}d}{\sqrt{\mu}}$.

In Figure \ref{AG_GD_SC}, we consider the rate-robustness trade-off for a strongly convex function with {$\mu=1$} and $L=20$ using the upper bounds we derived in this section for both GD and AG (that are applicable to any dimension $d$), where we report the normalized robustness level $\mathcal{J}/d$ in the $y$-axis versus the convergence rate in the $x$-axis -- for AG we plotted two curves associated with the two approaches detailed above: the first one (marked in red color) uses the explicit bound and 
the second one (marked in yellow color) is based on a grid search. We use the solver CVX \cite{cvx} to solve the 4-dimensional SDPs given in \eqref{eq-rbar-grid}. We select a grid of size $|I_1|=|I_2| = 30$ on the parameter space for $(\alpha,\beta)$ determined by Proposition \ref{prop-stability-set-ag}. We observe that the grid search and the explicit upper bound approach give similar results on this numerical example for AG, especially when the rate is close to 1. We also see that the robustness upper bound for GD obtained from our approach (marked in blue color) is worse than the upper bound we developed for AG. To see how the AG bounds are comparable with the exact expressions we developed for quadratics, first we note that by Lemma \ref{lem-H2-cvx}, any quadratic function \Sml~admits the robustness upper bound\footnote{This is a worst-case bound for a quadratic \Sml, tighter bounds are available if all the eigenvalues of its Hessian matrix were known or estimated beyond the eigenvalues $\mu$ and $L$, see Proposition \ref{prop-h2-nesterov}.} 
$$\mathcal{J}_{max}(\alpha,\beta) := (d-1)\max \left[ \uab(\mu), \uab(L)\right] + \min \left[ \uab(\mu), \uab(L)\right]$$
and this bound can be achieved for some choices of $f$. For large $d$, we have clearly $ \mathcal{J}_{max}(\alpha,\beta) / d \approx \max \left[ \uab(\mu), \uab(L)\right]$. In Figure \ref{AG_GD_SC}, we plot the latter quantity versus the convergence rate (marked in purple color) to demonstrate the rate-robustness curve for AG in the case of quadratic objective functions. We observe from Figure \ref{AG_GD_SC} that our bounds for the quadratic case are tighter than those for general strongly convex functions as expected.
\begin{figure}[ht!]
  \centering
    \includegraphics[width=0.4\textwidth]{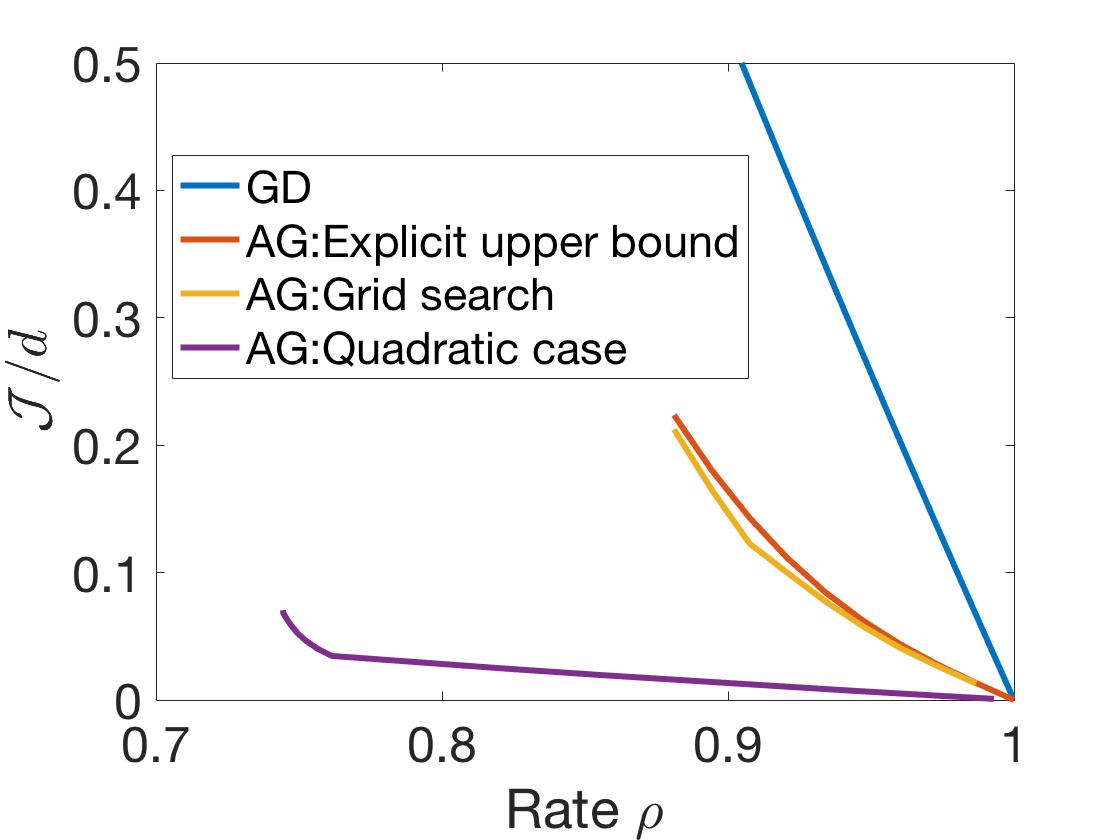}
  \caption{Rate-robustness trade-off for GD and AG algorithms using derived upper bounds and comparing it with the quadratic result.\label{AG_GD_SC}}
\end{figure}

\section{Asymptotic stability of the optimum with respect to perturbations}\label{sec:stability}
Our discussion so far has focused on the robustness  of first-order methods with respect to random noise in the gradients, which we quantify by $\mathcal{J}$ defined in \eqref{eq:robust_def}. Our robustness measure $\mathcal{J}$ is based on the $H_2$ norm of an associated linear dynamical system. It is well known that the $H_2$ norm of a dynamical system is closely related to the \emph{asymptotic stability} of the equilibrium (which is characterized by the optimal solution $x^*$ to \eqref{main-opt-prob} in our setup) in the sense that it quantifies how quickly the system can converge back to the equilibrium 
if it is unsettled from its equilibrium in the direction of a coordinate 
\cite{zhou1996robust}.
More specifically, for each $i\in\{1,\ldots, d\}$, let $\{x_k^i\}_{k\geq 0}$ be the iterate sequence corresponding to \eqref{noisy-dyn_sys2:main} whenever $\{w_k\}_k = \delta[k] e_i$ for $k\geq 0$ where $e_i$ is the $i$-th basis vector, i.e., we 
perturb the system from its equilibrium with an impulse input in the direction of 
$e_i$. 
Let
\begin{equation}\label{h2_impulse}
\mathcal{J}_* := \sum_{i=1}^d \| x^i - x^* \|_2^2 
\end{equation}
where $\| x^i - x^* \|_2$ is the $l_2$ norm of the sequence $\{x_k^i - x^*\}_k$. It is worth noting that 
$\{x_k^i\}_k$ is the same as the iterate sequence of the noiseless system \eqref{dyn_sys:main} with initial state 
$\xi^*+Be_i$, $D=0_d$ and $\phi(\cdot)=\nabla f(\cdot)$.

We next motivate this definition from another perspective. Recall the alternative robustness measure $\mathcal{J}'=\limsup_k \E[\|x_k-x^*\|^2]$ we briefly discussed in Section~\ref{sec-dyn-system-intro}; 
for a linear system it is known that $\mathcal{J'} =  \mathcal{J}_*$ \cite{zhou1996robust}. In other words, our explicit formulas and bounds for $\mathcal{J'}$, given in Appendix~\ref{iterate_case} of the Supplementary Material, translate immediately to $\mathcal{J}_*$ for optimizing {\it quadratic} functions. The definition in \eqref{h2_impulse} also extends to the case when $f$ is not necessarily quadratic or equivalently when the system \eqref{noisy-dyn_sys2:main} is nonlinear; however, for non-linear systems there is no known explicit formula that relates $\mathcal{J'}$ to $\mathcal{J}_*$ \cite{Stoorvogel93}. We refer to the quantity $\mathcal{J}_*$ as \emph{perturbation stability} of the first-order algorithm in consideration; 
{indeed, it measures how sensitive the underlying optimization algorithm is to the initialization around the optimal solution --- it also quantifies how strongly the iterate sequence is attracted to the optimal solution once they are close.} 

In particular, given $A$, $B$, and $C$ defining the first-order optimization algorithm, we assume there exists $X\in\mathbb{S}^{m+d}$ such that the MI in \eqref{eq:LMI} holds for some $P\in\mathbb{S}^m_+$ and $\rho\in(0,1)$. Moreover, we assume that when $\Gamma=0$ and $\sigma=0$, there exists a constant $c\geq 0$ independent of $\xi_0$ such that the same $\rho$ and $X$ satisfy the dissipation inequality in \eqref{eq:energy_dissipation} for every $k \geq 0$; hence, it follows from \eqref{eq:LMI} and \eqref{eq:energy_dissipation} along with \eqref{storage_contraction} that
{
\begin{align}
\label{eq:general_result-impulse}
    V_{P,c}(\xi_k)\leq \rho^{2k} V_{P,c}(\xi_0),\quad \forall~k\geq 0.
\end{align}}%
Since \Sml, whenever $P\in\mathbb{S}^m_{+}$ and/or $c>0$, \eqref{eq:general_result-impulse} implies that the error signal $\{\|x_k^i-x_*\|^2\}_k$ decays geometrically and is therefore summable; and as a consequence, $\mathcal{J}_*$ in \eqref{h2_impulse} well-defined.
\begin{lemma}
For GD, the following bound holds for all $\alpha\in(0,2/L)$
{
\begin{align}
\label{eq:GD_impulse}
    \mathcal{J}_*(\alpha)\leq \frac{\alpha^2d}{1-\rho^2_{GD}(\alpha)},
\end{align}}%
where $\rho_{GD}(\cdot)$ is defined in \eqref{eq:rho_GD}. Moreover, for AG, given  $\alpha\in(0,1/L]$, setting $\beta(\alpha)=\frac{1-\sqrt{\alpha\mu}}{1+\sqrt{\alpha\mu}}$, the perturbation stability can be bounded as $\mathcal{J}_*(\alpha)\leq \frac{\alpha^2d}{\sqrt{\alpha\mu}}(1+\kappa)$.%
\end{lemma}
\begin{proof}
Recall that $\{x_k^i\}_k$ is the same as the iterate sequence of the noiseless system \eqref{dyn_sys:main} with initial state 
$\xi^*+Be_i$. Hence,  Proposition~\ref{general_gd_cont} with $\sigma=0$ implies that
{
\begin{equation}
\label{eq:dterministic_rate}
\norm{x^i_{k}-x^*}^2 \leq \rho^{2k} \norm{x^i_0-x^*}^2,\quad \forall\ k\geq 0.
\end{equation}
for some $\rho\in(0,1)$ and for any $1 \leq i \leq d$ and stepsize $\alpha\in(0,2/L)$.
Thus, $\sum_{k=0}^\infty\norm{x^i_k-x^*}^2\leq \frac{1}{1-\rho^2}\norm{x^i_0-x^*}^2$,} which implies $\norm{x^i-x^*}^2\leq \frac{\alpha^2}{1-\rho^2}$ for all $i=1,\ldots, d$ since $B=-\alpha I_d$ for GD. Therefore, we have $\mathcal{J}_*(\alpha)\leq \frac{\alpha^2d}{1-\rho^2}$.
Moreover, given any stepsize $\alpha\in(0,2/L)$ for GD, using \eqref{eq:rho_GD}, which is the smallest $\rho$ value for which \eqref{eq:dterministic_rate} holds, we obtain \eqref{eq:GD_impulse}. On the other hand, for AG, using Corollary~\ref{agd_main_bounds_2} with $\sigma=0$ and the fact that $\frac{\mu}{2}\norm{x_k-x^*}^2\leq f(x^k)-f^*$, we get $\norm{x^i_k-x^*}^2\leq \rho_{AG}^{2k}(\norm{x_0^i-x^*}^2+\tfrac{2}{\mu}(f(x^i_0)-f^*))$ for $k\geq 0$ and $i=1,\ldots, d$, where we used $x_0^i=x_{-1}^i=x^*+Be_i$ for $i=1,\ldots,d$. Thus,
\begin{align*}
    \mathcal{J}_*(\alpha)\leq \frac{1}{1-\rho_{AG}^2(\alpha)}\Big(\Tr(B^\top B)+\sum_{i=1}^d \tfrac{2}{\mu}(f(x^i_0)-f^*))\Big)\leq \frac{\alpha^2d}{\sqrt{\alpha\mu}}(1+\kappa),\quad \forall~\alpha\in(0,1/L].
\end{align*}
\end{proof}
For $\alpha\in(0,1/L]$ since $\rho_{GD}(\alpha)=1-\alpha\mu$, \eqref{eq:GD_impulse} implies that $\mathcal{J}_*(\alpha)\leq \frac{\alpha^2 d}{\alpha\mu(2-\alpha\mu)}=\mathcal{O}(\alpha)$ for GD. For quadratic \Sml~such that $\mu=L$, 
{as we discussed in Footnote \ref{J_GD_iterates},}
$J'(\alpha)=\frac{\alpha d}{\mu(2-\alpha\mu)}$. 
Since $J_*(\alpha)=J'(\alpha)$ for quadratics, 
$\mathcal{O}(\alpha)$ dependence of \eqref{eq:GD_impulse} for $\alpha\in(0,1/L]$ is tight.
We also note that, $\mathcal{O}(\alpha^{3/2})$ bound on $\mathcal{J}_*(\alpha)$ for AG implies that for sufficiently small stepsize $\alpha$, AG possesses better perturbation stability than GD.
\section{Numerical Experiments}
Our first set of experiments concern a further study of Example \ref{example} for comparing AG and GD in terms of performance. 
In the leftmost plot of Figure \ref{numeric_1}, we vary the trade-off parameter from $\tau=0$ to $\tau=\infty$ for AG and plot the robustness level $\mathcal{J}(\alpha_*(\tau),\beta_*(\tau))$ versus the rate $\rho(\alpha_*(\tau),\beta_*(\tau))$ corresponding to the optimal parameters $(\alpha_*(\tau),\beta_*(\tau))$, we also plot the analogous curve for GD (the same curve from Figure \ref{fig-GD}) to compare. We observe that for the same achievable convergence rate, the optimized AG parameters lead to more robust algorithms compared to the optimized GD algorithms as AG has an additional parameter $\beta$ to optimize robustness over. This shows that AG can improve GD in terms of both convergence rate and robustness at the same time when gradients are subject to white noise. This result is in contrast with the deterministic gradient error setting in \cite{devolder2014first}, 
which shows that GD performance degrades gracefully while AG may accumulate error.
Therefore, our results suggest that AG algorithms can tolerate random noise better than deterministic noise to preserve their accelerated rates, which is also inline with the theoretical findings of \cite{CGZ-momentum-wass}. 
Also it is interesting to note that the popular choice of parameters (blue and red dots), as well as the parameters that lead to the optimal (fastest) rate (green and purple dots) lie on curves that trade robustness with rate in an optimal fashion.
\begin{figure}[ht!]
  \centering
  \begin{minipage}[t]{0.31\textwidth}
    \includegraphics[width=\textwidth]{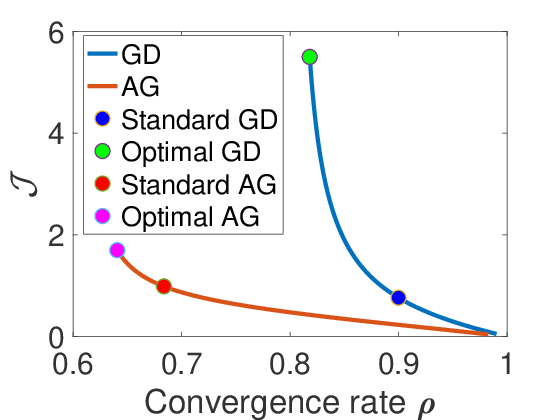}
  \end{minipage}%
  ~
  \begin{minipage}[t]{0.31\textwidth}
    \includegraphics[width=\textwidth]{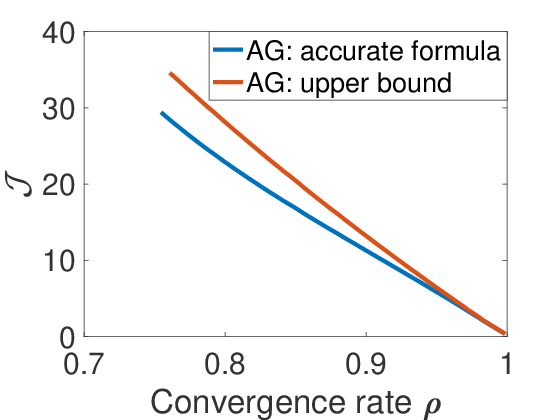}
  \end{minipage}%
  ~
  \begin{minipage}[t]{0.31\textwidth}
    \includegraphics[width=\textwidth]{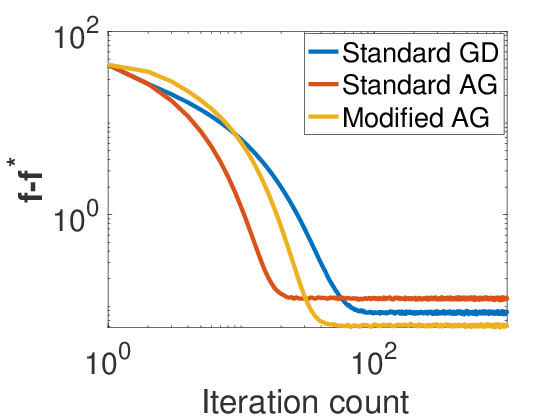}
  \end{minipage}
  \caption{\label{numeric_1}Left: Robustness $\mathcal{J}$ as a function of the convergence rate for GD and AG. Middle: Comparison of the convergence rate and robustness obtained by solving \eqref{opt-pbm-ag} versus its approximation \eqref{opt-pbm-ag-ub} for $d=100$. Right: Tuned AG can be both faster and more robust than GD.}
\end{figure}


Next, we illustrate the tightness of our upper bound $\bar{\mathcal{J}}(\alpha,\beta)$ provided in \eqref{eqn-h2sq-acc-grad-ub} to the (true) robustness level $\mathcal{J}(\alpha,\beta)$. This upper bound results in a small scale optimization problem \eqref{opt-pbm-ag-ub} that allows trading-off robustness and the convergence rate in a way that computationally tractable, even in high dimensions. 
The middle plot of Figure \ref{numeric_1} shows the convergence rate and robustness obtained by solving \eqref{opt-pbm-ag} versus solving \eqref{opt-pbm-ag-ub}. The objective is a random quadratic function in dimension $d=100$ with parameters $\mu=0.1, L=1$. Our results show that for any trade-off parameter $\tau$ our upper bound is within a factor of $1.2$ of true parameters, illustrating the accuracy of this approximation to the optimal parameters for different levels of robustness, especially the approximation is more accurate when the trade-off parameter is larger (in which case the convergence rate is closer to 1). We obtain quantitatively similar results 
repeating this experiment with other randomly generated quadratic functions.

Next, we illustrate our framework to trade-off robustness and convergence rate on a quadratic optimization problem, similar to the one considered in \cite{Hardt-blog} where it is shown that AG algorithms with standard choice of parameters have difficulty to handle random gradient noise. We consider the quadratic  function $f(x) = \frac{1}{2}x^\top Q x + b^T x + \delta \|x\|^2 $ in dimension $d=100$ where $Q$ is the Laplacian of a cyclic graph, $\delta=0.1$ is a regularization parameter to make the problem strongly convex and $b$ is a random vector. As it can be seen in the rightmost plot of Figure \ref{numeric_1}, we show that when properly modified, AG can be both faster and more robust in comparison with GD. 

\begin{figure}[t!] 
  \centering
  \begin{minipage}[t]{0.475\textwidth}
    \includegraphics[width=\textwidth]{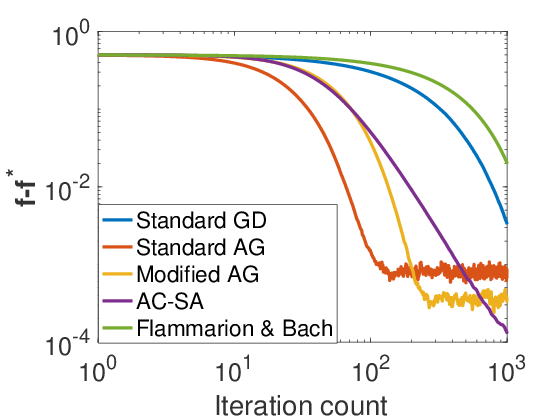}
  \end{minipage}%
 ~
    \begin{minipage}[t]{0.47\textwidth}
    \includegraphics[width=\textwidth]{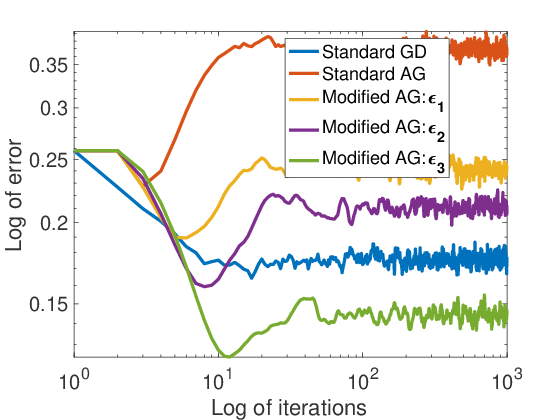}
  \end{minipage}
  \caption{\label{fig-Compare} Left: Comparison of the other algorithms with Modified AG. Right: Tuned AG for the regularized logistic regression problem.}
\end{figure}
In the leftmost plot of Figure \ref{fig-Compare}, we compare the tuned AG with 
other algorithms 
such as AC-SA \cite{Lan2012} and the Flammarion-Bach algorithm \cite{flammarion2015averaging}. For this purpose, we consider the same quadratic test problem from \cite{flammarion2015averaging} in dimension $d=20$, where the eigenvalues of its Hessian $Q$ are set equal to $\lambda_i = i^2$ for $i=1,2,...,20$. Our results show that modified AG can trade robustness with the convergence rate successfully and can improve upon AC-SA and Flammarion-Bach algorithm on this example. 

Finally, we validate our results for strongly convex and smooth functions by choosing function $f$ to be a regularized logistic loss. We synthesize a random matrix $M \in \mathbb{R}^{2000\times 100}$ and a random vector $w\in \R^{100}$ and compute $y= \mbox{sign}(Mw)$ as the output of the classifier. The goal is to recover $w$ using regularized logistic regression when the gradient of the logistic loss 
is corrupted with additive Gaussian noise. The plot on the right panel of Figure \ref{fig-Compare} shows the behavior of tuned AG after solving the optimization problem \eqref{general_AG-problem} with three different $\epsilon$ values $0<\epsilon_1 < \epsilon_2 < \epsilon_3$ in comparison with standard $AG$ and $GD$. 
As predicted by our theory, AG performs better than GD and the asymptotic suboptimality decreases as $\varepsilon$ gets larger. 

\section{Conclusion} 	
We consider the gradient descent (GD) and accelerated gradient (AG) methods for optimizing strongly convex functions. We developed a computationally tractable framework to design their parameters in a way to trade between two conflicting performance measures: the convergence rate and the robustness to additive white noise in the gradient computations measured in terms of final asymptotic variance of the algorithm output. For strongly convex quadratics, we show that 
this robustness measure is equal to the $H_2$ norm of a  dynamical system associated to the optimization algorithm and give an explicit characterization of this quantity. Our results show that for the same achievable rate, AG can always be tuned to be more robust. Similarly, for the same robustness level, we show that AG can be tuned to be always faster than GD. We also give fundamental lower bounds on the achievable robustness level for gradient descent for a given achievable rate. We show how our analysis can be extended to smooth strongly convex functions and we derive upper bounds on the robustness measures for GD and AG.
\section{Acknowledgments}
The work of Necdet Serhat Aybat is partially supported by NSF Grant CMMI-1635106. Alireza Fallah is partially supported by Siebel Scholarship. Mert G\"urb\"uzbalaban acknowledges support from the
grants NSF DMS-1723085 and NSF CCF-1814888.
\bibliographystyle{plain}
\bibliography{main}
\newpage
\begin{appendix}
\section{First-order optimality conditions for the objective \texorpdfstring{$F_\tau(\alpha)$}{Lg}}
\begin{proposition}\label{prop-gd-optimality}
There exists an optimizer $\alpha_*(\tau)$ to the minimization problem \eqref{opt-pbm-gd}. Furthermore, any optimizer is either $\alpha_*(\tau) = 2/(\mu+L)$ or it satisfies one of the following two conditions:
{
\begin{numcases}{}
\frac{\alpha^2}{2} \sum_{i=1}^d{\frac{1}{(2-\alpha \lambda_i)^2}} + \frac{\tau(\alpha \mu-1)}{\mu(2-\alpha \mu)^2}=0 \quad \mbox{and} \quad |1-\al \mu| > |1-\al L|. \label{poly-eqn-gd_1}\\
\frac{\alpha^2}{2} \sum_{i=1}^d{\frac{1}{(2-\alpha \lambda_i)^2}} + \frac{\tau(\alpha L-1)}{L(2-\alpha L)^2}=0 \quad \mbox{and} \quad |1-\al \mu| < |1-\al L|. \label{poly-eqn-gd_2}
\end{numcases}}%
Therefore, by examining the values of $F$ at the points that satisfy this equality and inequality constraints, we can determine the optimal stepsize $\al^*(\tau)$.
\end{proposition}
\begin{proof}
To solve \eqref{opt-pbm-gd} explicitly, note that 
\begin{equation*}
\frac{1}{1-\rho^2} = \begin{cases}
	\frac{1}{1 - (1-\al\mu)^2} & \mbox{if} 
    \quad |1-\al \mu| > |1-\al L|,    \\
 \frac{1}{1 - (1-\al L)^2}   &  \mbox{if} \quad |1-\al \mu| \leq |1-\al L|.  
\end{cases}
\end{equation*}
The optimal $\al^*$ cannot be attained on the boundary points of the interval $[0,2/L]$ as 
$F$ is not finite 
at these points. Therefore, it suffices to solve the optimization problem over the open interval $(0,2/L)$ where $F$ is differentiable with respect to $\al$ except when $|1-\al \mu| = |1-\al L|$, i.e. when $\al = 2/(\mu+L)$. For $\al^* \neq 2/(\mu+L)$, we can write-down the first-order conditions of optimality
$\tfrac{\partial F}{\partial \al}=0$ 
which leads to \eqref{poly-eqn-gd_1} and \eqref{poly-eqn-gd_2}.
\end{proof}

\section{Proof of Proposition \ref{prop-stability-set-ag}}
In the light of the formula \eqref{rho_nest} that characterizes $\rho(A_Q)$, the closure of the stability set $\calS$ admits the representation
$\calS = \calS_\mu \cap \calS_L$ where for $\lambda \in \{\mu, L\}$ we define 
\begin{equation}\label{def-S-intersection}
\mathcal{S_\lambda} = \left\{ (\al,\be) ~:~ \rho_\lambda(\al,
\be) \leq 1, \al\geq 0, \be\geq 0 \right\} \subset \R^2.
\end{equation}
We first write $\calS_\lambda$ as a union of two disjoint sets depending on the signature of $\Delta_\lambda$:
$\calS_\lambda = \calS_{\lambda,1} \cup  \calS_{\lambda,2}$
where 
\begin{eqnarray} \calS_{\lambda,1} = \calS_\lambda \cap \{(\al,\be)~:~  \Delta_\lambda \leq 0\}, \quad
\calS_{\lambda,2} = \calS_\lambda \cap \{(\al,\be)~:~  \Delta_\lambda > 0\}.
\end{eqnarray}
It follows from the definition of $\Delta_\lambda$ in \eqref{def-delta-lambda} that $\Delta_\lambda \leq 0$ if and only if
$0 \leq 1-\al\la \leq \frac{4\be}{(1+\be)^2}$;
and when this condition holds, $\rho_\la = \sqrt{\be(1-\al\la)} \leq 1$ if and only if
$ 0 \leq 1-\al\la \leq \frac{1}{\beta}. $
Therefore,
{
\begin{equation} \label{S_l_1}
\calS_{\la,1} = \{ (\al,\be):~0 \leq 1-\al\la \leq \min\{\frac{1}{\beta}, \frac{4\beta}{(1+\beta)^2} \}\}.
\end{equation}}%
We next focus on $\calS_{\la,2}$. Note that $\Delta_\lambda \geq 0$ if and only if
{
\begin{equation}\label{eqn-delta-nonneg}
1-\al\la \leq 0 \quad \mbox{or}\quad 1-\al\lambda \geq \frac{4\be}{(1+\be)^2}.
\end{equation}}%
If \eqref{eqn-delta-nonneg} is satisfied, then $\rho_\la \leq 1$ if and only if
$\frac{1}{2}(1+\be)(1-\al\la)\mbox{sign}(1-\al\lambda) + \frac{1}{2}\sqrt{\Delta_\la} \leq 1$.
There are two cases: 

1) $\Dla > 0$  and $1-\al \la < 0$: In this case, $\rho_\lambda \leq 1$ if and only if
    $\sqrt{\Delta_\la} \leq 2 - (1+\be) \cla$,
where $\cla=-(1-\al\la) > 0$. By squaring both sides, this is if and only if,  
	$ \Delta_\la \leq (2 - (1+\be) \cla)^2 \quad \mbox{and} \quad 2 - (1+\be) \cla \geq 0 $
The first inequality holds if 
 $ \cla = -(1-\al\la) \leq \tfrac{1}{2\be+1}$
whereas the second inequality holds if 
  $ \cla = -(1-\al\la) \leq \tfrac{2}{\be+1}.$
The first inequality is more binding, if it holds the second inequality holds too. Therefore,
\begin{equation} \label{S_l_2_1}
\{(\al,\be):~ 0 \leq -(1-\al\la) \leq \frac{1}{2\be + 1} \} \subset \calS_{\la,2}.
\end{equation}

2) $\Dla >0$ and $1-\al \la > 0$: In this case, 
 $\rho_\lambda \leq 1$ if and only if
    $ \sqrt{\Delta_\la} \leq 2 - (1+\be) \dla$ 
where $\dla:=-\cla = (1-\al\la) > 0$. After squaring both sides, this is if and only if 
	\begin{equation}\label{ineq-s-lam-2} \Delta_\la \leq (2 - (1+\be) \dla)^2 \quad \mbox{and} \quad 2 - (1+\be) \dla \geq 0 \end{equation}
where the first inequality simplifies to $1 \geq \dla$. \eqref{ineq-s-lam-2} along with \eqref{eqn-delta-nonneg} means
$\tfrac{4\beta}{(1+\beta)^2}\leq 1-\al\la \leq \min\{1, \tfrac{2}{1+\beta}\}$ 
which implies $\beta \leq 1$; therefore,
{
\begin{equation} \label{S_l_2_2}
\{(\al,\be):~ \frac{4\beta}{(1+\beta)^2} \leq 1-\al\la \leq \frac{2}{1+\beta} \} \subset \calS_{\la,2}.
\end{equation}}%
Merging \eqref{S_l_1}, \eqref{S_l_2_1}, and \eqref{S_l_2_2} yields
{
\begin{equation} \label{S_l}
S_\lambda = \left\{(\al,\be):~ 1-\al \la \in \left[-\frac{1}{1+2\beta}, \min\{\frac{1}{\beta}, \frac{2}{1+\beta} \}\right]\right\}.
\end{equation}}%
To complete the proof, due to the representation \eqref{def-S-intersection}, it suffices to compute the intersection $\calS_\mu \cap \calS_L$. There are several cases to consider depending on the value of $\alpha$:

1) First, consider $\alpha \in [0, \frac{1}{L}]$. In this case $1 - \al \mu \geq 1- \al L \geq 0$, and hence \eqref{S_l} implies
$1 - \al \mu \leq \frac{2}{1+\be}$ if $\beta \leq 1$ whereas $1 - \al \mu \leq \frac{1}{\be}$ if $\beta \geq 1$.
Nevertheless, if $\beta \leq 1$ then $\frac{2}{1+\beta} \geq 1$, so the first case always holds; 
hence, 
$(1 - \al \mu)\beta \leq 1$.

2) Now, assume $\alpha \in [\frac{1}{L}, \min\{\frac{2}{L}, \frac{1}{\mu}\}]$. Then $1 - \al \mu \geq 0 \geq 1- \al L$, and thus \eqref{S_l} yields
$
1- \al L \geq -\frac{1}{1+2\beta}, \quad 1- \al \mu \leq \min\{\frac{1}{\beta}, \frac{2}{1+\beta}\}
$
where the second inequality again simplifies to $(1 - \al \mu)\beta \leq 1$.

3) The last possible case happens when $\mu \geq \frac{L}{2}$ , and so $\alpha \in [\frac{1}{\mu}, \frac{2}{L}]$ is possible. In this case $1 - \al L \leq 1- \al \mu \leq 0$, and so using \eqref{S_l}, we just need to check 
$
1- \al L \geq -\frac{1}{1+2\beta}
$
Considering all these cases along with the fact that \eqref{S_l} shows $\alpha$ cannot be greater than $\frac{2}{L}$ completes the proof.
\section{Proof of Proposition \ref{prop-h2-nesterov}\label{sec-appendix-h2-nesterov}}
Similar to the analysis for 
GD, we can assume without loss of generality that $Q$ is diagonal. The proof is also similar. Consider $U \Lambda U^\top$ be the eigenvalue decomposition of $Q$. Then $A_Q$ in \eqref{A_Q_nest} can be written as
{
\begin{align}
&A_Q = \tilde{U} A_\Lambda \tilde{U}^\top,\quad \hbox{where} \label{A_la_nest}\\
&\tilde{U} = \begin{bmatrix} 
	U & 0_d\\
    0_d & U
\end{bmatrix},
A_\Lambda = \begin{bmatrix} 
	(1+\be)(\id-\al \Lambda) & -\be(\id-\al \Lambda)\\
    \id & 0_d
\end{bmatrix}. \label{UTilde}
\end{align}}%
Replacing $A_Q$ from \eqref{A_la_nest} in Lyapunov equation \eqref{eqn-lyapunov} implies
\begin{equation}
\tilde{U} A_\Lambda \tilde{U}^\top X \tilde{U} A_\Lambda^\top \tilde{U}^\top  - X + B B^\top = 0.
\end{equation}
Multiplying by $\tilde{U}$ and $\tilde{U}^\top$ from right and left respectively yields
\begin{equation} \label{A_Lambda_lyapunov}
A_\Lambda \tilde{U}^\top X \tilde{U} A_\Lambda^\top - \tilde{U}^\top X \tilde{U} + B B^\top = 0
\end{equation}
where we used the fact that $B$ in \eqref{nest_ABCD} has the property that $\tilde{U}^\top B B^\top \tilde{U} = B B^\top$. 

Equation \eqref{A_Lambda_lyapunov} shows that $\tilde{U}^\top X \tilde{U}$ satisfies the Lyapunov equation when $A_Q$ is replaced by $A_\Lambda$. Next, we show that after after substituting $Q$ by $\Lambda$, the robustness $\mathcal{J}(\al, \be)$ would stay the same;
i.e., $H_2^2(A_\Lambda, B, \sqrt{\frac{1}{2}}\Lambda^{1/2}T) = H_2^2(A_Q, B, RT)$, where $R=\sqrt{\frac{1}{2}} \Lambda^{1/2}U^\top$. To show this, note that from \eqref{formula-h2_1}, we have 
\begin{equation*}
\begin{split}
H_2^2(A_\Lambda, B, \sqrt{\frac{1}{2}}\Lambda^{1/2}T) &= \frac{1}{2}\Tr((\Lambda^{1/2}T)\tilde{U}^\top X \tilde{U} (\Lambda^{1/2}T)^\top)\\
&=\frac{1}{2} \Tr((\Lambda^{1/2}T\tilde{U}^\top) X (\Lambda^{1/2}T\tilde{U}^\top)^\top) = \Tr((RT)X (RT)^\top),
\end{split}
\end{equation*}
where the last equality is true as $T \tilde{U}^\top = U^\top T$ for $T$ and $U$ given in 
\eqref{quad_func_err_3} and \eqref{UTilde}. 
This result 
completes the proof of our claim that we can assume $Q$ is diagonal. For simplicity we will continue our analysis with $A_Q$, assuming its a diagonal matrix.

Let $P_\pi$ be the permutation matrix associated with the permutation $\pi$ over the set $\{1,2,...,2d\}$ that satisfies $\pi(i)=2i-1$ for $1 \leq i \leq d$ and $\pi(i)=2(i-d)$ for $d+1 \leq i \leq 2d$.
It is well-known that permutation matrices satisfy
$P_\pi^{-1} = P_\pi^\top = P_{\pi^{-1}}$;
therefore, multiplying Lyapunov equation \eqref{eqn-lyapunov} by $P_\pi$ and $P_\pi^\top$ from left and right, respectively, 
leads to
\begin{equation} \label{lyapunov_per}
(P_\pi A_Q P_\pi^\top) Y (P_\pi A_Q^\top P_\pi^\top) - Y + P_\pi BB^\top P_\pi^\top = 0,
\end{equation}
where $Y=P_\pi X P_\pi^\top$. 
It follows from \eqref{A_Q_nest} that 
{
\begin{equation*}
P_\pi A_Q P_\pi^\top=
\diag([T_i]_{i=1}^d),\quad\hbox{and}\quad 
{
T_i =  \begin{bmatrix}
(1+\be)(1-\al \lambda_i) & -\be(1-\al \lambda_i) \\
1      & 0 
\end{bmatrix}},\ i=1,\ldots,d
\end{equation*}}%
and
$0 < \mu=\lambda_1 \leq \lambda_2 \leq ... \leq \lambda_d=L$ are the eigenvalues of $Q$. Since
{
$B B^T=  
\begin{bmatrix} 
\alpha^2 \id & 0_d \\
0_d     & 0_d
\end{bmatrix}$,}%
$P_\pi BB^\top P_\pi^\top$ is a $2d$ by $2d$ diagonal matrix with $\alpha^2$ on entries $(1,1),(3,3),...,$ $(2d-1,2d-1)$ and zero elsewhere. Hence, $Y$ that solves \eqref{lyapunov_per} is a block diagonal matrix in the form:
{
$Y=\diag([Y_i]_{i=1}^d)
$},
where
{
$Y_i =  \begin{bmatrix} 
y_i^u & y_i^o  \\
y_i^o & y_i^d  
\end{bmatrix}$
}%
satisfies the equality
{
\begin{align*}
\begin{bmatrix} 
(1+\be)(1-\al \lambda_i) &  -\be(1-\al \lambda_i)\\
1      & 0 
\end{bmatrix} Y_i 
      			\begin{bmatrix} 
(1+\be)(1-\al \lambda_i) & 1 \\
-\be(1-\al \lambda_i)    & 0 
\end{bmatrix}  - Y_i 
      			+ \begin{bmatrix}
	\alpha^2   & 0 \\
     0 &   0
\end{bmatrix} = 0
\end{align*}}%
for all $i=1,\dots,d$. This is equivalent to the linear system:
{
\begin{equation*}
\begin{bmatrix} (1+\be)^2(1-\al\lambda_i)^2 - 1 & -2\be(1+\be)(1-\al\lambda_i)^2 & \be^2(1-\al\lambda_i)^2 \\
(1+\be)(1-\al\lambda_i) & -1-\be(1-\al\lambda_i) & 0 \\
1 & 0 & -1 
\end{bmatrix} 
\begin{bmatrix} y_i^u  \\ y_i^o \\ y_i^d
				\end{bmatrix} 
= 
\begin{bmatrix}
	-\al^2 \\ 0 \\ 0
\end{bmatrix}.
\end{equation*}}%
Solving this system of equations, we obtain:
{
\begin{equation} \label{Y_solutions}
\begin{split}
y_i^u = y_i^d &=\al\frac{1+\be(1-\al \lambda_i)}{\lambda_i (1-\be(1-\al\lambda_i))(2+2\be-\al\lambda_i(1+2\be))},\\
y_i^o &= \frac{\alpha^2(1+\be)(1-\al \lambda_i)}{\al \lambda_i (1-\be(1-\al\lambda_i))(2+2\be-\al\lambda_i(1+2\be))}.
\end{split}
\end{equation}}%
The $\mathcal{J}(\al, \be)$ can be computed using
{
\begin{equation}
\begin{split}
\Tr((\sqrt{\frac{1}{2}}Q^{1/2}T) X (\sqrt{\frac{1}{2}}Q^{1/2}T)^\top) 
= \frac{1}{2} \Tr(P_\pi T^\top Q T P_\pi^\top Y).
\end{split}
\end{equation}}%
The matrix $P_\pi T^\top Q T P_\pi^\top$ is block diagonal with 
$2\times 2$ matrices 
$
\begin{bmatrix} 
\lambda_i & 0 \\
0 & 0
\end{bmatrix}
$
on its diagonal. Therefore, using \eqref{Y_solutions}, the robustness measure $\mathcal{J}(\al, \be)$ is equal to
{
\begin{equation}
\frac{1}{2} \sum_{i=1}^d \al \frac{1+\be(1-\al \lambda_i)}{(1-\be(1-\al\lambda_i))(2+2\be-\al\lambda_i(1+2\be))}.
\end{equation}}%
\section{Convexity of \texorpdfstring{$\uab(\lambda)$}{Lg}}
We next show that 
$\uab(\lambda)$ appearing in the definition of the $\mathcal{J}(\al,\be)$ for the AG algorithm is convex with respect to $\lambda$.
\begin{lemma}\label{lem-H2-cvx} Let $(\alpha,\beta) \in \calS$ where $\calS$ is the stability region of the dynamical system representation of AG given by \eqref{stability_def}. The function $\uab(\lambda)$ defined by \eqref{eq-s-lam} is convex on the interval $[\mu,L]$. 
\end{lemma}
\begin{proof}  
The function $\uab(\lambda)$ can be written in terms of $\tla:=\be(1-\al \lambda)$ as follows
\begin{equation}\label{u_tilde}
q_\lambda(\tla) = \frac{\al}{2}\frac{1+\tla}{(1-\tla)(1+\tla\gamma)}
\end{equation}
where $\gamma:=2+\frac{1}{\beta}$. It follows from \eqref{stability_def} that for $(\alpha,\beta) \in \calS$,
$
1 \geq \tla \geq \frac{-1}{\gamma}    
$,
and thus both terms in denominator of \eqref{u_tilde} are positive. Note that $\tla$ is linear, and since the composition of a convex functions with a linear function is convex, it suffices to show $q_\lambda(\tla)$ is a convex function of $\tla$ over domain $[\frac{-1}{\gamma},1]$. To show this, we simply compute the second derivative of \eqref{u_tilde} with respect to $\tla$. After doing some algebra, 
{
\begin{equation}
\frac{d}{d \tla^2}q_\lambda(\tla) = \frac{\al}{2}\left(\frac{2(\gamma^3-\gamma^2)}{(\gamma+1)(\gamma \tla+1)^3}+\frac{4}{(\gamma+1)(1-\tla)^3}\right)  
\end{equation}}%
which is non-negative as $\gamma \geq 2$ and $\tla \in [\frac{-1}{\gamma},1]$. This completes the proof.
\end{proof}
\section{Defining rate and robustness based on iterates}\label{iterate_case}
In this supplementary file, we first recall how an alternative robustness measure can be defined based on the distance of the iterates to the optimal solution instead of the robustness measure $\mathcal{J}$ we introduced in the main text based on the asymptotic expected suboptimality in function values. Here, we focus on the iterate sequence $\{x_k\}_k$ 
to characterize the notions of rate and robustness. First we consider the case that $f$ is a quadratic function in the form of $f(x) = \tfrac{1}{2} x^\top Q x - p^\top x + r$. Using \eqref{iterates_rate} along with the relation $x_k = T \xi_k$ with $T$ defined in Section \ref{rate_robustness_quad}, for both GD and AG the sequence $\{\E[x_k]\}_k$ goes to zero with rate $\rho(A_Q)$.

However, due to the noise injected at each step, the limit of the sequence $\{x_k\}$ will oscillate around the optimal solution with a non-zero variance. 
Thus, a natural metric to measure robustness is then to study the \emph{asymptotic normalized variance} by considering the limit $\mathcal{J}' = \lim_{k\to\infty} \frac{1}{\sigma^2}\E[\norm{x_k-x^*}^2]$. 
Similar line of argument as Lemma \ref{H_2_quad_equi} shows that this limit exists and is in fact equal to $H_2^2(A_Q,B,T).$ This quantity can be viewed as the \emph{robustness to noise in terms of iterates} because it is 
equal to the ratio of the power of the iterates to the power of the input noise, measuring how much a system amplifies input noise. In particular, the smaller this measure is, the more robust the system is under additive random noise.

The robustness $\mathcal{J}'$ can be evaluated precisely for GD and AG method same as what we did in Section \ref{quad_case} for $\mathcal{J}$. For GD method with constant stepsize $\al \in (0,2/L)$, the robustness to noise in terms of iterates is denoted as $\mathcal{J}'(\al)$ to show the dependence to $\al$. The following proposition, which can be proved similar to Proposition \ref{quad_GD_thm}, shows the explicit characterization of $\mathcal{J}'(\al)$.
\begin{proposition} \label{quad_GD_thm_2}
Let $f$ be a quadratic function of the form $f(x) = \tfrac{1}{2} x^\top Q x - p^\top x + r$. Consider the GD iterations given by \eqref{gradient_update} with constant stepsize $\alpha \in (0,2/L)$ . Then the robustness of the GD method in terms of iterates is given by
\begin{align}\label{GD_J_iterates}
\mathcal{J}'(\alpha) =  \alpha^2\sum_{i=1}^d\frac{1}{1-(1-\alpha \lambda_i)^2} = \alpha \sum_{i=1}^d\frac{1}{\lambda_i (2-\alpha \lambda_i)} ,
\end{align}
where $0<\mu = \lambda_1 \leq \lambda_2 \leq ... \lambda_d=L$ are the eigenvalues of $Q$.
\end{proposition}
For AG, with constant stepsize $\al$ and momentum parameter $\beta$, we denote the robustness to noise in terms of iterates as $\mathcal{J}'(\alpha, \be)$. The following theorem, which can be proved similar to Proposition \ref{prop-h2-nesterov}, provides an explicit formula for $\mathclap{J}'(\al,\be)$ in terms of the eigenvalues of $Q$.
\begin{proposition}
Let $f$ be a quadratic function of the form $f(x) = \tfrac{1}{2} x^\top Q x - p^\top x + r$. Consider the AG iterations given by \eqref{nest:main} with parameters $(\alpha,\beta) \in \mathcal{S}$ . Then the robustness of the AG method in terms of iterates is given by
\begin{align}
	\mathcal{J}'(\alpha,\beta) = \sum_{i=1}^d u'_{\alpha,\beta}(\lambda_i)
\end{align}
where $\mu = \lambda_1 \leq \lambda_2 \leq \dots \leq \lambda_d = L$ are the eigenvalues of $Q$ and
	{
    \beq u'_{\alpha,\beta} (\lambda) \triangleq\al \frac{1+\be(1-\al \lambda)}{\lambda (1-\be(1-\al\lambda))(2+2\be-\al\lambda(1+2\be))}.
    \eeq
    }
\end{proposition}
As discussed in Section \ref{quad_case}, the $\mathcal{J}'(\al,\be)$ admits a tractable upper bound in the form of
$	\mathcal{J}'(\al, \be) \leq d  \max(u'_{\alpha,\beta} (\mu), u'_{\alpha,\beta} (L))$
which only depends on $\mu$ and $L$.

We can also extend the definitions of rate and robustness in terms of iterates to the case that $f\in S_{\mu,L}(\mathbb{R}^d)$. Note that in general the sequence $\{x_k\}_k$ might not converge to the optimal solution in expectation (see \cite{bachGD}). Using the family of Lyapunov functions $V_{P,c}(\xi)$, it can be shown that, similar to \eqref{main_contraction}, the following inequality holds for both GD and AG with properly chose parameters
\begin{equation}\label{main_contraction_iterates}
\E[\|x_k-x^*\|^2] \leq \rho^{2k} \psi'_0 + \sigma^2 R' \quad k \geq 1
\end{equation}
where $0<\rho<1$ is the same $\rho$ as \eqref{main_contraction} and also $\psi'_0$ and $R'$ are non-negative numbers and depend on algorithm parameters and initial point $x_0$. For instance, Proposition \ref{general_gd_cont} implies that  \eqref{main_contraction_iterates} holds for GD, i.e., for all $k \geq 0$, 
\begin{align}
&\E[\norm{x_{k}-x^*}^2] \leq \rho(\al)^{2k} \norm{x_0-x^*}^2 + \sigma^2 R'(\al),\quad \hbox{where} \label{eq:GD_result_iterates}\\
&R'(\al) \triangleq \frac{\alpha^2 d}{1-\rho(\al)^2}.    \label{GD_R_iterates}
\end{align}
Similarly, we can derive \eqref{main_contraction_iterates} for AG by using Proposition \ref{AG_mainresult}.

Finally, we show that the $R'(\al)$ is a tight bound for $\mathcal{J}'$.
For quadratic \Sml, $\mathcal{J}'$ can be written in closed form as in \eqref{GD_J_iterates}, 
Note that 
if $\lambda_i = \mu = L$ for $i=1,\ldots, d$, then this quantity is equal to $R'(\al)$.
\end{appendix}

\end{document}